\pdfoutput=1
\RequirePackage{ifpdf}
\ifpdf 
\documentclass[pdftex]{sigma}
\else
\documentclass{sigma}
\fi

\numberwithin{equation}{section}

\usepackage{amsmath}

\usepackage{amssymb}
\usepackage{stmaryrd}
\usepackage{amsthm}
\usepackage[mathscr]{eucal}
\usepackage{amscd}
\usepackage[all]{xy}
\usepackage{url}

\newtheorem{Thm}{Theorem}[section]

\usepackage{enumerate}

\usepackage{comment}

\usepackage{color}

\newtheorem{Lem}[Thm]{Lemma}
\newtheorem{Prop}[Thm]{Proposition}

\newtheorem{Conj}[Thm]{Conjecture}

\newtheorem*{Thm*}{Theorem}

\newtheorem*{Conj*}{Conjecture}

{\theoremstyle{definition}

\newtheorem*{Def*}{Definition}
\newtheorem{Eg}[Thm]{Example}
\newtheorem{Rem}[Thm]{Remark}
\newtheorem{Def}[Thm]{Definition}
}

\newcommand{\kk}{\Bbbk}
\newcommand{\Z}{\mathbb{Z}}
\newcommand{\N}{\mathbb{N}}

\newcommand{\opname}[1]{\operatorname{#1}}

\newcommand{\pr}{\opname{pr}}

\newcommand{\op}{^{op}}

\newcommand{\supp}{\opname{supp}}

\renewcommand{\deg}{\opname{deg}}

\newcommand{\Hf}{{\frac{1}{2}}}

\newcommand{\Rm}[1]{{\longmapsto}}
\newcommand{\Lm}[1]{{\longmapsfrom}}

\newcommand{\cA}{{\mathcal A}}

\newcommand{\cC}{{\mathcal C}}

\newcommand{\cF}{{\mathcal F}}

\newcommand{\bL}{{\mathbf L}}

\newcommand{\tB}{{\widetilde{B}}}

\newcommand{\can}{L}
\newcommand{\clAlg}{{\cA}}
\newcommand{\qClAlg}{\cA_q}

\newcommand{\diag}{{\delta}}

\usepackage{tikz}
\usetikzlibrary{positioning,shapes,shadows,arrows,snakes}

\pgfdeclarelayer{edgelayer}
\pgfdeclarelayer{nodelayer}
\pgfsetlayers{edgelayer,nodelayer,main}

\tikzstyle{none}=[inner sep=0pt]
\tikzstyle{black box}=[draw=black, fill=black!25]
\tikzstyle{white box}=[draw=black, fill=white]
\tikzstyle{black circle}=[circle,draw=black!50, fill=black!25]
\tikzstyle{red circle}=[circle,draw=red!50, fill=red!25]
\tikzstyle{blue circle}=[circle,draw=blue!50, fill=blue!25]
\tikzstyle{green circle}=[circle,draw=green!50, fill=green!25]
\tikzstyle{yellow circle}=[circle,draw=yellow!50, fill=yellow!25]

\newcommand{\thistheoremname}{}
\newtheorem*{genericthm*}{\thistheoremname}
\newenvironment{namedthm*}[1]
  {\renewcommand{\thistheoremname}{#1}%
   \begin{genericthm*}}
  {\end{genericthm*}}

\renewcommand{\diag}{{d}}

\renewcommand{\can}{{\bL}}

\newcommand{\fv}{\opname{f}}
\newcommand{\ufv}{\opname{uf}}
\newcommand{\codeg}{\opname{codeg}}

\renewcommand{\qClAlg}{\clAlg_q}

\newcommand{\bQClAlg}{\overline{\clAlg}_q}

\newcommand{\Mc}{M^\circ}

\newcommand{\Nufv}{N_{\opname{uf}}}

\newcommand{\qUpClAlg}{\mathcal{U}_q}

\newcommand{\LP}{{\mathcal{LP}}}

\newcommand{\cRing}{\mathcal{P}}

\renewcommand{\diag}{d}

\newcommand{\mm}{{\mathbf{m}}}

\newcommand{\frg}{{\mathfrak{g}}}

\newcommand{\Inj}{\mathbf{I}}
\newcommand{\Proj}{\mathbf{P}}

\newcommand{\seq}{\overleftarrow{\mu}}
\newcommand{\seqnu}{\overleftarrow{\nu}}

\newcommand{\hLP}{\widehat{\mathcal{LP}}}
\newcommand{\tLP}{\widetilde{\mathcal{LP}}}

\newcommand{\envAlg}{\opname{U_q}}

\newcommand{\ow}{\overrightarrow{w}}

\newtheorem{Definition-Lemma}[Thm]{Definition-Lemma}

\newcommand{\qO}{{\opname{A_q}}}

\newcommand{\up}{{\opname{up}}}

\newcommand{\midAlg}{{\opname{A}}}

\newcommand{\dCan}{\opname{B}^{\up}}

\renewcommand{\op}{\mathrm{op}}
\newcommand{\coTrop}{\phi^{\op}}

\newcommand{\hmu}{\widehat{\mu}}

\begin{document}
\allowdisplaybreaks

\newcommand{\arXivNumber}{2004.12466}

\renewcommand{\thefootnote}{}

\renewcommand{\PaperNumber}{122}

\FirstPageHeading

\ShortArticleName{An Analog of Leclerc's Conjecture for Bases of Quantum Cluster Algebras}

\ArticleName{An Analog of Leclerc's Conjecture\\ for Bases of Quantum Cluster Algebras\footnote{This paper is a~contribution to the Special Issue on Cluster Algebras. The full collection is available at \href{https://www.emis.de/journals/SIGMA/cluster-algebras.html}{https://www.emis.de/journals/SIGMA/cluster-algebras.html}}}

\Author{Fan QIN}

\AuthorNameForHeading{F.~Qin}

\Address{School of Mathematical Sciences, Shanghai Jiao Tong University,\\ Shanghai 200240, People's Republic of China}
\Email{\href{mailto:fgin11@sjtu.edu.cn}{fgin11@sjtu.edu.cn}}
\URLaddress{\url{https://sites.google.com/site/qinfanmath/}}

\ArticleDates{Received May 14, 2020, in final form November 13, 2020; Published online November 27, 2020}

\Abstract{Dual canonical bases are expected to satisfy a certain (double) triangularity property by Leclerc's conjecture. We propose an analogous conjecture for common triangular bases of quantum cluster algebras. We show that a weaker form of the analogous conjecture is true. Our result applies to the dual canonical bases of quantum unipotent subgroups. It~also applies to the $t$-analogs of $q$-characters of simple modules of quantum affine algebras.}


\Classification{13F60}

\renewcommand{\thefootnote}{\arabic{footnote}}
\setcounter{footnote}{0}

\section{Introduction}

\subsection{Background}

\textbf{Dual canonical bases and cluster theory.}
Let $\frg$ denote a Kac--Moody algebra with a~sym\-metrizable Cartan
datum, and $\envAlg=\envAlg(\frg)$ the corresponding quantized enveloping
algebra, where $q$ is not a root of unity. The negative (or~positive)
part $\envAlg^{-}$ of $\envAlg$ possesses the famous canonical bases
\cite{Kashiwara90,Kas:crystal,Lusztig90,Lusztig91}.
The corresponding dual basis $\dCan$ also has fascinating properties
and is related to the theory of total positivity \cite{Lusztig96}.

Fomin and Zelevinsky invented cluster algebras as a combinatorial
framework to understand the total positivity \cite{Lusztig96} and
the dual canonical bases $\dCan$. We refer the reader to the survey
\cite{Keller08Note} for further details of cluster algebras.

Let there be given any Weyl group element $w\in W$. Then the dual
canonical basis $\dCan$ of $\envAlg^{-}$ restricts to a basis $\dCan(w)=\dCan \cap \qO[N_{-}(w)]$
for the quantum unipotent subgroup $\qO[N_{-}(w)]$, see \cite{Kimura10}.
Notice that, if $\frg$ is a finite-dimensional semi-simple Lie algebra,
then $\qO[N(w_{0})]$ agrees with $\envAlg^{-}$, where $w_{0}$ denotes
the longest element in $W$.

Thanks to previous works (such as \cite{BerensteinFominZelevinsky05,BerensteinZelevinsky05, GeissLeclercSchroeer10, GeissLeclercSchroeer11,goodearl2020integral, GY13}), it is known that the quantum
unipotent subgroup $\qO[N_{-}(w)]$ is a (partially compactified)
quantum cluster algebra $\bQClAlg(t_{0})$, where the initial seed
$t_{0}=t_{0}(\ow)$ is constructed using a reduced word $\ow$ of~$w$. By Fomin and Zelevinsky~\cite{FominZelevinsky02}, the dual
canonical basis $\dCan(w)$ is expected to contain all quantum cluster
monomials, which was formulated as the quantization conjecture for
Kac--Moody cases in~\cite{Kimura10}. This conjecture has been verified
for acyclic cases by \cite{KimuraQin14} based on \cite{HernandezLeclerc09,
Nakajima09}, for symmetric semisimple cases and partially for
symmetric Kac--Moody cases by~\cite{qin2017triangular}, for all symmetric
Kac--Moody cases by \cite{Kang2018}, and recently, for all symmetrizable
Kac--Moody cases by \cite{qin2020bases}.

\textbf{Leclerc's conjecture.}
A basis element $b\in\dCan\subset\envAlg^{-}$ is said to be real
if $b^{2}\in q^{\Z}\dCan$. Leclerc proposed the following conjecture
regarding the multiplication by a real element of $\dCan$, which
is analogous to Kashiwara crystal graph operator.

\begin{Conj}[Leclerc's conjecture {\cite[Conjecture 1]{Leclerc03}}]\label{conj:leclerc}
Assume that $b_{1}$ is a real element of~$\dCan$. Then, for any
$b_{2}\in\dCan$ such that $b_{1}b_{2}\notin q^{\Z}\dCan$, the expansion
of their product on $\dCan$ takes the form
\begin{gather}
b_{1}b_{2} = q^{h}b'+q^{s}b''+\sum_{c\neq b',b''}\gamma_{b_{1},b_{2}}^{c}c,
\label{eq:Leclerc_conj}
\end{gather}
where $b'\neq b''$, $h<s\in\Z$, $\gamma_{b_{1},b_{2}}^{c}\in q^{h+1}\Z[q]\cap q^{s-1}\Z\big[q^{-1}\big]$.
\end{Conj}

This conjecture was proved by \cite{Kang2018} for symmetric Kac--Moody
cases using quiver Hecke algebras \cite{KhovanovLauda08,KhovanovLauda08:III,Rouquier08}.

\textbf{Common triangular bases.}
For a large class of quantum (upper) cluster algebras (called injective-reachable,
see Section~\ref{subsec:Injective-reachability}), the author introduced
the triangular basis $\can^{t}$ for any chosen seed~$t$ in~\cite{qin2017triangular}.
The basis is characterized by a triangular property with respect to
the dominance order on the degrees of its basis elements (Section~\ref{subsec:Pointedness}), whence the name ``triangular''. It is
unique if it exists, and it can be constructed via Lusztig Lemma for
Kazhdan--Lusztig type bases \cite[Section~6.1]{qin2020bases}. Notice
that the triangular basis $\can^{t}$ depends on the seed~$t$.

In \cite[Definition 6.1.1]{qin2017triangular}, the author further
considered the common triangular basis $\can$, such that it gives
rise to the triangular bases $\can^{t}$ for all seeds $t$ and the
basis elements have well-behaved degrees under mutations (Sections~\ref{subsec:Tropical-transformations} and~\ref{subsec:Bases-with-different}).
The common triangular basis, if exists, contains all quantum cluster
monomials and verifies the Fock--Goncharov dual basis conjecture~\cite{FockGoncharov06a,FockGoncharov09}.
It~turns out the dual canonical bases of $\qO[N_{-}(w)]$ give rise
to the common triangular bases for the corresponding quantum cluster
algebras \cite{kashiwara2019laurent,qin2017triangular,qin2020bases}.
Also, the collections of the simple modules in monoidal categorification
of cluster algebras also provide examples of the common triangular
bases \cite{cautis2019cluster,Kang2018,kashiwara2019laurent,qin2017triangular}.
In this view, the common triangular bases suggest a~generalization
of the dual canonical bases in cluster theory, and their existence
also implies the possible existence of monoidal categorifications.

\subsection{Main results}

By \cite{kashiwara2019laurent,qin2017triangular,qin2020bases},
after localization and rescaling, the dual canonical basis $\dCan(w)$
agrees with the common triangular basis of the corresponding quantum
cluster algebra in the sense of \cite{qin2017triangular}. Correspondingly,
we formulate the following analog of Leclerc's conjecture.

\begin{Conj}[{Conjecture~\ref{conj:analog_leclerc}}]\label{conj:intro_analog_leclerc}
Conjecture $\ref{conj:leclerc}$ is true if we replace the dual canonical
basis by the common triangular basis.
\end{Conj}

Recall that the quantum cluster monomials provide a subset of the
real elements in the dual canonical basis $\dCan(w)$ (we conjecture
that all real elements take this form, see Conjecture~\ref{conj:d_can_basis_reachability}).
Our first main result is the following weaker form of the analogous
conjecture.

\begin{Thm}[{Theorem~\ref{thm:weaker_analog_leclerc}}]\label{thm:intro_weaker_analog_leclerc}
Conjecture $\ref{conj:intro_analog_leclerc}$ is true for the real basis
elements corresponding to quantum cluster monomials.
\end{Thm}

Theorem~\ref{thm:intro_weaker_analog_leclerc} implies a triangularity
property for the $t$-analogs of $q$-characters of simple modu\-les
of quantum affine algebras, see Theorem~\ref{thm:analog-leclerc-quantum-aff}.
By \cite[Corollary 4.12]{kashiwara2020monoidal}, as a stronger version
of Theorem~\ref{thm:analog-leclerc-quantum-aff}, Conjecture~\ref{conj:intro_analog_leclerc}
holds true for the quantum cluster algebras in this case, see Remark~\ref{rem:categorical_leclerc}.

Our second main result follows as a consequence of Theorem~\ref{thm:intro_weaker_analog_leclerc}.

\begin{Thm}[{Theorem~\ref{thm:weaker_leclerc}}]
If we consider the dual canonical basis $\dCan(w)$ of the quantum
unipotent subgroup $\qO[N_{-}(w)]$, then Conjecture~$\ref{conj:leclerc}$
holds true for the real elements corresponding to quantum cluster
monomials.
\end{Thm}

In order to study the analog of Leclerc's conjecture and prove Theorem
\ref{thm:intro_weaker_analog_leclerc}, we will consider not only
triangularity with respect to degrees but also triangularity with
respect to codegrees. Correspondingly, we introduce the notion of
double triangular bases (Definition~\ref{def:double_triangular_basis}),
such that both the degrees and the codegrees of the basis elements
are well-behaved. In fact, the two terms $b'$, $b''$ in (the analog
of) Leclerc's conjecture are determined by the codegree and the degree
respectively, see Lemma~\ref{lem:leclerc_conj_one_seed}.

As the third main result, we verify that the common triangular basis
elements have well-behaved degrees and codegrees.

\begin{Thm}[{Theorem~\ref{thm:common_tri_to_double}}]
If the common triangular basis exists, then it gives rise to the double
triangular basis for any seed. Moreover, it is \emph{compatibly copointed}.
\end{Thm}

It is worth remarking that, if the cluster algebra is categorified
by a rigid monoidal category, then degrees and codegrees are related
to the two different ways of taking the dual objects in the category,
see \cite{kashiwara2019laurent}.

\subsection{Contents}

In Section~\ref{sec:Basics-of-cluster}, we briefly review basic notions
in cluster theory needed by this paper.

In Section~\ref{subsec:Pointedness}, we review notions and techniques
introduced and studied by \cite{qin2017triangular,qin2019bases}
such as dominance orders, (co)degrees and (co)pointed functions. In
Section~\ref{subsec:Tropical-transformations}, we define tropical
transformation for codegrees in analogous to that for degrees. In
Section~\ref{subsec:Injective-reachability}, we review the notion
of injective-reachability, and define the sets of distinguished functions
$\Inj^{t}$, $\Proj^{t}$ for seeds $t$, and we present some properties
of injective-reachability and these distinguished functions.

In Section~\ref{sec:Bidegrees-and-bases}, we define various bases
whose degrees or codegrees satisfy certain properties. In~particular,
we introduce the notion of double triangular bases. We discuss the
relation between double triangular bases and (common) triangular bases.
We prove that common triangular bases have good properties on their
codegrees (Theorem~\ref{thm:common_tri_to_double}).

In Section~\ref{sec:Main-results}, we propose an analog of Leclerc's
conjecture for common triangular bases (Conjecture~\ref{conj:analog_leclerc})
and show a weaker form holds true (Theorem~\ref{thm:weaker_analog_leclerc}).
We discuss its consequences for modules of quantum affine algebras
(Theorem~\ref{thm:analog-leclerc-quantum-aff}, Remark~\ref{rem:categorical_leclerc}).
We deduce that the weaker form is satisfied by the dual canonical
bases of $\qO[N_{-}(w)]$ (Theorem~\ref{thm:weaker_leclerc}).

\section{Basics of cluster algebras\label{sec:Basics-of-cluster}}

We briefly review notions in cluster theory necessary for this paper
following \cite{qin2017triangular,qin2019bases,qin2020bases}.
A~rea\-der unfamiliar with cluster theory is referred to
\cite{BerensteinZelevinsky05,Keller08Note} for background materials.

Denote $\kk=\Z\big[q^{\pm\Hf}\big]=\Z[v^{\pm}]$, where $v=q^{\Hf}$ is a
formal parameter. Define $\mm=v^{-1}\Z\big[v^{-1}\big]$. Notice that we have
a natural bar involution $\overline{(\ )}$ on $\kk$ which sends
$v$ to $v^{-1}$. Let $(\ )^{T}$ denote the matrix transposition
and $[\ ]_{+}$ denote the function $\max(0,\ )$.

\subsection{Seeds}

Fix a finite set of vertices $I$ and its partition $I=I_{\ufv}\sqcup I_{\fv}$
into the unfrozen and frozen vertices.

Let there be given a quantum seed \smash{$t=\big(\tB(t),\Lambda(t),(X_{i}(t))_{i\in I}\big)$},
where $X_{i}(t)$ are indeterminates, the integer matrices $\tB(t)=(b_{ij}(t))_{i\in I,j\in I_{\ufv}}$
and $\Lambda(t)=(\Lambda_{ij}(t))_{i,j\in I}$ form a compatible pair,
i.e. there exists some diagonal matrix $D=\mathrm{diag}(\diag_{k})_{k\in I_{\ufv}}$
with strictly positive integer dia\-go\-nals, such that $\tB(t)^{T}\Lambda(t)=\begin{pmatrix}
D & 0\end{pmatrix}$. $X_{i}(t)$ are called the $i$-th $X$-variables or quantum cluster
variables associated to $t$, $\tB(t)$ the $\tB$-matrix, $\Lambda(t)$
the $\Lambda$-matrix, and $B(t):=(b_{ij}(t))_{i,j\in I_{\ufv}}$
the principal part of $\tB(t)$ or the $B$-matrix.

\begin{Lem}[\cite{BerensteinZelevinsky05}]\quad
\begin{enumerate}\itemsep=0pt
\item[$(1)$] We have $\diag_{i}b_{ik}(t)=-\diag_{k}b_{ki}(t)$ for $i,k\in I_{\ufv}$.

\item[$(2)$] The matrix $\tB(t)$ is of full rank $|I_{\ufv}|$.
\end{enumerate}
\end{Lem}

Define the following lattices (of column vectors):
\begin{gather*}
\Mc(t) = \oplus_{i\in I}\Z f_{i}(t)\simeq\Z^{I},
\\
\Nufv(t) = \oplus_{k\in I_{\ufv}}\Z e_{k}(t)\simeq\Z^{I_{\ufv}},
\end{gather*}
where $f_{i}(t)$, $e_{k}(t)$ denote the $i$-th and $k$-th unit
vectors respectively. Denote
\begin{gather*}
\Nufv^{\geq0}(t)=\oplus_{k\in I_{\ufv}}\N e_{k}(t)\simeq\N^{I_{\ufv}}.
\end{gather*}

Define the linear map $p^{*}\colon\Nufv(t)\rightarrow\Mc(t)$ such that
$p^{*}n=\tB(t)n$. Let $\lambda$ denote the bilinear form on $\Mc(t)$
such that
\begin{gather*}
\lambda(g,g') = g^{T}\Lambda(t)g'.
\end{gather*}

\begin{Lem}
For any $i\in I$, $k\in I_{\ufv}$, we have $\lambda(f_{i}(t),p^{*}e_{k}(t))=-\delta_{ik}\diag_{k}$.
\end{Lem}

The group algebra of $\Mc(t)$ is the Laurent polynomial ring $\kk[\Mc(t)]:=\kk[X(t)^{m}]_{m\in\Mc(t)}=\kk[X_{i}(t)^{\pm}]_{i\in I}$
with the usual addition and multiplication $(+,\cdot)$, where we
denote the Laurent monomial $X(t)^{m}=\prod_{i\in I}X_{i}(t)^{m_{i}}$ for $m=\sum m_{i}f_{i}(t)$.

The quantum Laurent polynomial ring (also called the quantum torus)
$\LP(t)$ associated to~$t$ is defined as the commutative algebra
$\kk[\Mc(t)]$ further endowed with the twisted product~$*$:
\begin{gather*}
X(t)^{m}*X(t)^{m'} = v^{\lambda(m,m')}X(t)^{m+m'}.
\end{gather*}
By the algebraic structure on $\LP(t)$, we mean $(+,*)$ unless otherwise
specified.

The monomials $X(t)^{m}$, $m\in\N^{I}$, are called the quantum cluster
monomials associated to~$t$. The Laurent monomials $X(t)^{m}$, $m\in\N^{I_{\ufv}}\oplus\Z^{I_{\fv}}$,
are called the localized quantum cluster monomials associated to $t$.

Define the $Y$-variables to be $Y_{k}(t):=X(t)^{p^{*}e_{k}(t)}$,
$k\in I_{\ufv}$. Denote $Y(t)^{n}=X(t)^{p^{*}n}$ for $n\in\Nufv(t)$.

We also define $\cF(t)$ to be the skew field of fractions of $\LP(t)$.

For simplicity, we often omit the symbol $t$ when there is no confusion.

\subsection{Mutations}

For any $k\in I_{\ufv}$, we have an operation $\mu_{k}$ called mutation
in the direction $k$ which gives us a~new seed $t'=\mu_{k}t=\big(\tB(t'),\Lambda(t'),(X_{i}(t'))_{i\in I}\big)$,
where $X_{i}':=X_{i}(t')$ are indeterminates. See \cite{BerensteinZelevinsky05}
for precise definitions of $\tB(t')$, $\Lambda(t')$. Recall that
we have $\mu_{k}^{2}t=t$.

Given any initial seed $t_{0}$, we let $\Delta_{t_{0}}^{+}$ denote
the set of seeds obtained from $t_{0}$ by iterated mutations. Then
we have $\Delta_{t_{0}}^{+}=\Delta_{t}^{+}$ if $t\in\Delta_{t_{0}}^{+}$.
Throughout this paper, we will always work with seeds from the same
set $\Delta^{+}=\Delta_{t_{0}}^{+}$, where the initial seed $t_{0}$
is often omitted for simplicity.

For simplicity, denote $t=\big(\tB,\Lambda,(X_{i})\big)$ and $t'=\big(\tB',\Lambda',(X_{i}')\big)$.

Denote $v_{k}=v^{\diag_{k}}$. Recall that there is an algebra isomorphism
$\mu_{k}^{*}\colon\cF(t')\simeq\cF(t)$ called the mutation birational
map, such that we have
\begin{gather*}
\mu_{k}^{*}(X_{k}') =v^{\lambda(f_{k},\sum_{j\in I}[-b_{jk}]_{+}f_{j})}X_{k}^{-1}*\big(X^{\sum_{j\in I}[-b_{jk}]_{+}f_{j}}+v_{k}^{-1}X^{\sum_{i\in I}[b_{ik}]_{+}f_{i}}\big)
\end{gather*}
and, for $i\neq k$,
\begin{gather*}
\mu_{k}^{*}(X_{i}') = X_{i}.
\end{gather*}

Notice that we can also write $\mu_{k}^{*}(X_{k}')=X^{-f_{k}+\sum_{j\in I}[-b_{jk}]_{+}f_{j}}\cdot(1+Y_{k})$.
Recall that $(\mu_{k}^{*})^{2}$ is an~identity.

Let there be given any seed $t'=\seq_{t',t}t$, where $\seq_{t',t}=\seq=\mu_{k_{r}}\cdots\mu_{k_{2}}\mu_{k_{1}}$
is a sequence of~mutations (read from right to left). We define the
mutation birational map $\seq_{t',t}^{*}\colon\cF(t')\simeq\cF(t)$ as
the composition $\mu_{k_{1}}^{*}\cdots\mu_{k_{r}}^{*}$. It is known
that $\seq_{t',t}^{*}$ is independent of the choice for the mutation
sequence $\seq_{t',t}$ from $t$ to $t'$. Define a sequence of mutations
$\seq_{t,t'}=\mu_{k_{1}}\mu_{k_{2}}\cdots\mu_{k_{r}}$ correspondingly,
which is often denoted by $(\seq_{t',t})^{-1}$ for simplicity.

Notice that, if $i\in I_{\fv}$, we have $\seq_{t',t}^{*}X_{i}(t')=X_{i}(t)$
for all $t'\in\Delta^{+}$. Correspondingly, we call~$X_{i}(t')$,
$i\in I_{\fv}$, $t'\in\Delta^{+}$, the frozen variables, and denote
them by $X_{i}$ for simplicity. Define the set of frozen factors
to be $\cRing=\big\{X^{m}\,|\, m\in\Z^{I_{\fv}}\big\}$.

\subsection{Cluster algebras}

Let there be given a quantum seed $t\in\Delta^{+}$.

\begin{Def}
The (partially compactified) quantum cluster algebra $\bQClAlg(t)$
is defined to be the $\kk$-subalgebra of $\LP(t)$ generated by the
quantum cluster variables $\seq_{t',t}^{*}X_{i}(t')$, $i\in I$,
$t'\in\Delta^{+}$.

The (localized) quantum cluster algebra $\qClAlg(t)$ is defined to
be the localization of $\bQClAlg(t)$ at $\cRing$.

The upper quantum cluster algebra $\qUpClAlg(t)$ is defined to be
$\cap_{t'\in\Delta^{+}}\seq_{t',t}^{*}\LP(t')$.
\end{Def}

Recall that we have $\bQClAlg(t)\subset\qClAlg(t)\subset\qUpClAlg(t)$.
Moreover, for $t,t'\in\Delta^{+}$, we have $\seq_{t',t}^{*}\qUpClAlg(t')=\qUpClAlg(t)$,
$\seq_{t',t}^{*}\qClAlg(t')=\qClAlg(t)$, $\seq_{t',t}^{*}\bQClAlg(t')=\bQClAlg(t)$.
It is sometimes convenient to forget the sym\-bols~$t$,~$t'$ by viewing
$\seq_{t',t}^{*}$ as an identification.

\section{Dominance orders and pointedness}

In this section, we recall the notions and some basic results concerning
dominance orders and pointed functions from \cite{qin2017triangular,qin2019bases}. We also describe properties of codegrees and
copointed functions in analogous to those of degrees and pointed functions.

\subsection{Dominance orders and pointedness}\label{subsec:Pointedness}

Let there be given a quantum seed $t$.

\begin{Def}[dominance order] We denote $g'\preceq_{t}g$ if there exists some $n\in\Nufv^{\geq0}(t)$
such that $g'=g+p^{*}n$. In this case, we say $g'$ is dominated by~$g$, or $g'$ is inferior to $g$.
\end{Def}

The meanings of symbols $\prec_{t}$, $\succ_{t}$, $\succeq_{t}$ are given in the obvious way.

\begin{Lem}[\cite{qin2017triangular}]\label{lem:finite_interval}
For any $g,g'\in\Mc(t)$, there exist finitely many $g''$ such that
$g''\preceq_{t}g$ and~$g''\succeq_{t}g'$.
\end{Lem}

Notice that $\LP(t)$ has a subring $\kk[\Nufv^{\geq0}(t)]:=\kk[Y_{k}(t)]_{k\in I_{\ufv}}$.
Let $\widehat{\kk[Y_{k}(t)]_{k\in I_{\ufv}}}$ denote the completion
of $\kk[Y_{k}(t)]_{k\in I_{\ufv}}$ with respect to the maximal ideal
generated by $Y_{k}(t)$, $k\in I_{\ufv}$. The~formal completion
of $\LP(t)$ is defined to be{\samepage
\begin{gather*}
\hLP(t) = \LP(t)\otimes_{\kk[Y_{k}(t)]_{k\in I_{\ufv}}}\widehat{\kk[Y_{k}(t)]_{k\in I_{\ufv}}}.
\end{gather*}
Elements in $\hLP(t)$ will be called functions or formal Laurent
series.}

Similarly, we consider the subring $\kk\big[Y_{k}^{-1}(t)\big]_{k\in I_{\ufv}}$
of $\LP(t)$ and its completion $\widehat{\kk\big[Y_{k}^{-1}(t)\big]_{k\in I_{\ufv}}}$
with respect to the maximal ideal generated by $Y_{k}^{-1}(t)$, $k\in I_{\ufv}$.
We define the following completion of $\LP(t)$:
\begin{gather*}
\tLP(t) : =\LP(t)\otimes_{\kk[Y_{k}^{-1}(t)]_{k\in I_{\ufv}}}\widehat{\kk\big[Y_{k}^{-1}(t)\big]_{k\in I_{\ufv}}}.
\end{gather*}

By a formal sum, we mean a possibly infinite sum. Let $Z$ denote
a formal sum $Z=\sum_{m\in\Mc(t)}b_{m}X(t)^{m}$. Notice that it belongs
to $\hLP(t)$ (resp.~$\tLP(t)$) if and only if its Laurent deg\-ree
support $\supp_{\Mc(t)}Z=\{m\,|\, b_{m}\neq0\}$ has finitely many $\prec_{t}$-maximal
elements (resp.~finitely many $\prec_{t}$-minimal elements).

\begin{Def}[(co)degrees and (co)pointedness]
The formal sum $Z$ is said to have degree $g$ if $\supp_{\Mc(t)}Z$
has a unique $\prec_{t}$-maximal element $g$, and we denote $\deg^{t}Z=g$.
It is said to be pointed at $g$ or $g$-pointed if we further have
$b_{g}=1$.

The formal sum $Z$ is said to have codegree $\eta$ if $\supp_{\Mc(t)}Z$
has a unique $\prec_{t}$-minimal element~$\eta$, and we denote $\codeg^{t}Z=\eta$.
It is said to be copointed at $\eta$ or $\eta$-copointed if we further
have $b_{\eta}=1$.

Let there be given a set $S$. It is said to be $\Mc(t)$-pointed
if it takes the form $S=\{S_{g}\,|\, g\in\Mc(t)\}$, where $S_{g}$ are
$g$-pointed functions in $\hLP(t)$. Similarly, it is said to be
$\Mc(t)$-copointed, if it takes the form $S=\{S^{\eta}\,|\,\eta\in\Mc(t)\}$,
where $S^{\eta}$ are $\eta$-copointed functions in $\tLP(t)$.
\end{Def}

As in \cite{qin2019bases}, if $Z$ is both pointed at $g$ and copointed
at $\eta$, it is said to be bipointed at the bidegree $(g,\eta)$.

\begin{Def}[normalization]
Let $\cF(\kk)$ denote the fraction field of $\kk$. If $Z$ has degree
$g$, we~define its (degree) normalization in $\hLP(t)\otimes_{\kk}\cF(\kk)$
to be
\[
[Z]^{t}:=b_{g}^{-1}Z.
\]
Similarly, if $Z$ has codegree $\eta$, we define its codegree normalization
in $\tLP(t)\otimes_{\kk}\cF(\kk)$ to be:
\[
\{Z\}^{t}:=b_{\eta}^{-1}Z.
\]
\end{Def}

Let there be given a (possibly infinite) collection of formal sums
$Z_{j}$. Notice that their formal sum $\sum_{j}Z_{j}$ is well-defined
if, at each Laurent degrees, only finitely many of them have non-vanishing
coefficients.

\begin{Def}[degree triangularity]
A formal sum $\sum_{j}b_{j}Z_{j}$ of pointed elements $Z_{j}\in\hLP(t)$,
$b_{j}\in\kk$, is said to be degree $\prec_{t}$-unitriangular, or
$\prec_{t}$-unitriangular for short, if $\big\{\deg^{t}Z_{j}\,|\, b_{j}\neq0\big\}$
has a unique $\prec_{t}$-maximal element $\deg^{t}Z_{j_{0}}$ and
$b_{j_{0}}=1$. It is further said to be degree $(\prec_{t},\mm)$-unitriangular,
or $(\prec_{t},\mm)$-unitriangular for short, if we further have
$b_{j}\in\mm$ for $j\neq j_{0}$.
\end{Def}

\begin{Def}[codegree triangularity]
A formal sum $\sum_{j}b_{j}Z_{j}$ of copointed elements $Z_{j}\in\tLP(t)$,
$b_{j}\in\kk$, is said to be codegree $\succ_{t}$-unitriangular
if $\{\codeg^{t}Z_{j}\,|\, b_{j}\neq0\}$ has a unique $\prec_{t}$-minimal
element $\codeg^{t}Z_{j_{0}}$ and $b_{j_{0}}=1$. It is further said
to be codegree $(\succ_{t},\mm)$-unitriangular, if we further have
$b_{j}\in\mm$ for $j\neq j_{0}$.
\end{Def}

Notice that Lemma~\ref{lem:finite_interval} implies that a degree
$\prec_{t}$-unitriangular sum is a well-defined sum in~$\hLP(t)$
and, similarly, a codegree $\succ_{t}$-unitriangular sum is a well-defined
sum in $\tLP(t)$.

\begin{Lem}\label{lem:triangular_decomp}\quad
\begin{enumerate}\itemsep=0pt
\item[$(1)$] Let there be given a $\Mc(t)$-pointed
set $S$, then any pointed function $Z\in\hLP(t)$ can be written
uniquely as a $($degree$)$ $\prec_{t}$-unitriangular sum of elements
of $S$ {\rm\cite{qin2017triangular}}.

\item[$(2)$] Let there be given a $\Mc(t)$-copointed set $S$, then any copointed
element $Z\in\tLP(t)$ can be written uniquely as a codegree $\succ_{t}$-unitriangular
sum of elements of $S$.
\end{enumerate}
\end{Lem}

\begin{proof}$(1)$ is proved as in \cite[Lemma~3.1.10(i)]{qin2017triangular}, see
also \cite[Definition--Lemma~4.1.1]{qin2019bases}.
$(2)$~can be proved similarly, or we can deduce it from (1) by using
the map $\iota$ defined in~\eqref{eq:opposite_hLP}.
\end{proof}

In the cases of Lemma~\ref{lem:triangular_decomp}, we say $Z$ is
(degree) $\prec_{t}$-unitriangular to $S$ or codegree $\succ_{t}$-unitriangular
to $S$ respectively. It is further said to be (degree) $(\prec_{t},\mm)$-unitriangular
to $S$ or~codegree $(\succ_{t},\mm)$-unitriangular to $S$ respectively,
if its decomposition in $S$ has such properties.

\begin{Eg}[type $A_2$]\label{eg:A2_cluster}
Take $I=\{1,2\}=I_{\ufv}$ and $I_{\fv}=\varnothing$. Consider the
initial quantum seed~$t_{0}$ such that $B(t_{0}):=B:=\left(\begin{smallmatrix}
0 & -1\\1 & 0
\end{smallmatrix}\right)$, $\Lambda(t_{0}):=\Lambda:=B$ and denote its quantum cluster variables
by~$X_{1}$,~$X_{2}$ for simplicity. Then we have $Y_{1}=X^{f_{2}}=X_{2}$
and $Y_{2}=X^{-f_{1}}=X_{1}^{-1}$, where~$f_{i}$ denote the $i$-th
unit vector.

The corresponding quantum cluster algebra $\qClAlg$ has five quantum
cluster variables. It has the basis $\can$ consisting of the quantum
cluster monomials.

Recall that, in $\LP(t_{0})$, the commutative product is denoted
by $\cdot$ or is omitted, while the twisted product is denoted by
$*$. The non-initial quantum cluster variables can be written as
the following in $\LP(t_{0})$:
\begin{gather*}
P_{2} :=X^{f_{1}-f_{2}}+X^{-f_{2}}=X^{f_{1}-f_{2}}\cdot(1+Y_{2}),
\\
I_{2}:=P_{1} :=X^{-f_{2}}+X^{-f_{1}-f_{2}}+X^{-f_{1}}=X^{-f_{2}}\cdot(1+Y_{2}+Y_{1}\cdot Y_{2}),
\\
I_{1}:=X^{-f_{1}}+X^{-f_{1}+f_{2}}=X^{-f_{1}}\cdot(1+Y_{1}).
\end{gather*}
Then $P_{2}$, $P_{1}$, $I_{1}$ are bipointed at the bidegrees $(f_{1}-f_{2},-f_{2})$,
$(-f_{2},-f_{1})$, $(-f_{1},-f_{1}+f_{2})$ respectively.

Let us compute some degree normalized products in $\LP(t_{0})$:
\begin{gather*}
[X_{1}*I_{1}] =X_{1}*I_{1}=1+v^{-1}X_{2},
\\
[X_{1}*I_{2}] =v^{\Lambda_{12}}X_{1}*I_{2}=X^{f_{1}-f_{2}}+X^{-f_{2}}+v^{-1}\cdot1=P_{2}+v^{-1}\cdot1,
\\
[X_{2}*I_{1}] =v^{\Lambda_{21}}X_{2}*I_{1}=X^{f_{2}-f_{1}}+X^{-f_{1}+2f_{2}},
\\
[X_{2}*I_{2}] =X_{2}*I_{2}=1+v^{-1}X^{-f_{1}}+v^{-1}X^{-f_{1}+f_{2}}=1+v^{-1}I_{1}.
\end{gather*}
Notice that $v^{\Lambda_{21}}X_{2}*I_{1}$ is a quantum cluster monomial
of the seed $\mu_{1}(t_{0})$. Then the above normalized products
are degree $(\prec_{t},\mm)$-unitriangular to the basis $\can$.

Similarly, we compute some codegree normalized products in $\LP(t_{0})$:
\begin{gather*}
\{P_{1}*X_{1}\} =P_{1}*X_{1}=v^{-\Lambda_{21}}X^{f_{1}-f_{2}}+v^{-\Lambda_{21}}X^{-f_{2}}+1=v^{-1}P_{2}+1,\\
\{P_{1}*X_{2}\} =v^{\Lambda_{12}}P_{1}*X_{2}=v^{-1}+X^{-f_{1}}+X^{f_{2}-f_{1}}=v^{-1}\cdot1+I_{1},\\
\{P_{2}*X_{1}\} =v^{\Lambda_{21}}P_{2}*X_{1}=X^{2f_{1}-f_{2}}+X^{f_{1}-f_{2}},\\
\{P_{2}*X_{2}\} =P_{2}*X_{2}=v^{-1}X_{1}+1.
\end{gather*}
Notice that $v^{\Lambda_{21}}P_{2}*X_{1}$ is a quantum cluster monomial
of $\mu_{2}(t_{0})$. Then the above normalized products are codegree
$(\succ_{t},\mm)$-unitriangular to the basis $\can$.
\end{Eg}

\subsection{Tropical transformations and compatibility\label{subsec:Tropical-transformations}}

As before, let there be given seeds $t'=\seq t$, where $\seq=\seq_{t',t}$
is a sequence of mutations. Denote $\seq_{t,t'}=\seq_{t',t}^{-1}$.
Denote the $i$-th cluster variables associated to $t$ and $t'$
by $X_{i}$ and $X_{i}'$ respectively. Let $f_{i}$, $f_{i}'$ denote
the $i$-th unit vectors associated to $t$ and $t'$ respectively.

\begin{Def}[tropical transformation]
If $t'=\mu_{k}t$, $k\in I_{\ufv}$, we define the (degree) tropical
transformation $\phi_{t',t}\colon\Mc(t)\simeq\Mc(t')$ such that, for any
$g=(g_{i})_{i\in I}\in\Mc(t)\simeq\Z^{I}$, its image $\phi_{t',t}g=(g'_{i})_{i\in I}\in\Mc(t')\simeq\Z^{I}$
is given by
\begin{gather*}
g'_{i} = \begin{cases}
-g_{k}, & i=k,\\
g_{i}+b_{ik}[g_{k}]_{+}, & i\neq k,\quad b_{ik}\geq0,\\
g_{i}+b_{ik}[-g_{k}]_{+}, & i\neq k,\quad b_{ik}<0.
\end{cases}
\end{gather*}
\end{Def}

In general, we define the (degree) tropical transformation $\phi_{t',t}\colon\Mc(t)\simeq\Mc(t')$
as the composition of the tropical transformations for adjacent seeds
along the mutation sequence $\seq$ from~$t$ to~$t'$. By~\cite{gross2013birational},
$\phi_{t',t}$ is the tropicalization of certain birational maps between
the split algebraic tori associate to $t$, $t'$ and, consequently, independent
of the choice of $\seq$.

Recall that $\seq_{t,t'}^{*}X_{i}$ is a pointed Laurent polynomial
in $\LP(t')$ by \cite{DerksenWeymanZelevinsky09,
gross2018canonical,Tran09}.

\begin{Def}[degree linear transformation {\cite[Definition 3.3.1]{qin2019bases}}]
Define $\psi_{t',t}\colon\Mc(t)\simeq\Mc(t')$ to be the linear map such
that $\psi_{t',t}(f_{i})=\deg^{t'}\seq_{t,t'}^{*}X_{i}$.
\end{Def}

By \cite[Lemma~3.3.4]{qin2019bases}, the mutation map $\seq_{t,t'}^{*}\colon\cF(t)\simeq\cF(t')$
induces an injective algebra homomorphism $\hmu\colon\LP(t)\hookrightarrow\hLP(t')$.
It has the following property.

\begin{Lem}[{\cite[Lemma~3.3.7]{qin2019bases}}]
For any $m\in\Z^{I}$, $\hmu X^{m}$ is a well-defined function in
$\hLP(t')$ pointed at degree $\psi_{t',t}m$.
\end{Lem}

Moreover, for $Z\in\LP(t)\cap\seq_{t',t}^{*}\LP(t')$, we have $\hmu(Z)=\seq_{t,t'}^{*}Z$,
see \cite[Lemma~3.3.4]{qin2019bases}. Correspondingly, denote $\hmu$
by $\seq_{t,t'}^{*}$ for simplicity.

Consider the following set of Laurent polynomials
\begin{gather*}
\LP(t;t') : =\LP(t)\cap\seq_{t',t}^{*}\LP(t').
\end{gather*}
Then $\LP(t;t')$ is a $\kk$-algebra, such that $\seq_{t,t'}^{*}\LP(t;t')=\LP(t';t)$.

The following very useful result shows that certain mutation sequences
swap pointedness and copointedness, where $t[-1]$ is a shifted seed
constructed from a given seed $t$ (Definition~\ref{def:inj_reachable}).

\begin{Prop}[Swap {\cite[Propositions~3.3.9 and~3.3.10]{qin2019bases}}]\label{prop:swap_deg_codeg}\quad
\begin{enumerate}\itemsep=0pt
\item[$(1)$] For any $g,\eta\in\Mc(t)$, we have $\eta\preceq_{t}g$ if and
only if $\psi_{t[-1],t}\eta\succeq_{t}\psi_{t[-1],t}g$.

\item[$(2)$] Let there be given $Z\in\LP(t;t[-1])\subset\LP(t)$. Then $Z$
is $\eta$-copointed if and only if $\seq_{t,t[-1]}^{*}Z$ is $\psi_{t[-1],t}\eta$-pointed.
\end{enumerate}
\end{Prop}

\begin{Def}[compatibility]
If $Z$ belongs to $\LP(t;t')\subset\LP(t)$, then $Z$ is said to
be compatibly pointed at $t$, $t'$ if it is $g$-pointed for some $g\in\Mc(t)$,
and $\seq_{t,t'}^{*}Z$ is $\phi_{t',t}g$-pointed.

If $Z$ belongs to $\qUpClAlg(t)\subset\LP(t)$, then $Z$ is said
to be compatibly pointed at $\Delta^{+}$ if it is compatibly pointed
at $t$, $t'$ for any $t'\in\Delta^{+}$.

Let $S$ denote a set consisting of $g$-pointed functions $S_{g}\in\hLP(t)$
for distinct $g\in\Mc(t)$. If $S_{g}$ is compatibly pointed at $t$, $t'$
for each $g$, we say $S$ is compatibly pointed at $t$, $t'$, or the
pointed sets $S$ and $\seq_{t,t'}^{*}S$ are (degree) compatible.
\end{Def}

Thanks to \cite{DerksenWeymanZelevinsky09,gross2018canonical},
we know that any given cluster monomial is compatibly pointed at all
seeds. In particular, the degrees of its Laurent expansions at different
seeds are related by tropical transformations.

\begin{Def}[codegree tropical transformation]
For any seeds $t'=\mu_{k}t$, $k\in I_{\ufv}$, we define the codegree
tropical transformation $\coTrop_{t',t}\colon\Mc(t)\simeq\Mc(t')$ as such
that, for any $g=(g_{i})_{i\in I}\in\Mc(t)\simeq\Z^{I}$, its image
$\coTrop_{t',t}g=(g'_{i})_{i\in I}\in\Mc(t')\simeq\Z^{I}$ is given
by
\begin{gather*}
g'_{i} = \begin{cases}
-g_{k}, & i=k,\\
g_{i}-b_{ik}[g_{k}]_{+}, & i\neq k,\quad b_{ik}\leq0,\\
g_{i}-b_{ik}[-g_{k}]_{+}, & i\neq k,\quad b_{ik}>0.
\end{cases}
\end{gather*}
\end{Def}

In general, we define the codegree tropical transformation $\coTrop_{t',t}\colon\Mc(t)\simeq\Mc(t')$
as the composition of the codegree tropical transformations for adjacent
seeds along the mutation sequence~$\seq$ from $t$ to $t'$.

Let us justify our definition of the codegree tropical transformation.

To any given seed $t=\big(\tB,\Lambda,(X_{i})_{i\in I}\big)$, we associate
the opposite seed defined to be $\iota(t):=t^{\op}:=\big({-}\tB,-\Lambda,(X_{i})_{i\in I}\big)$.
Then \cite[Lemma~2.2.5]{qin2019bases} implies that, for any mutation
sequence $\seq$, we have $(\seq t)^{\op}=\seq(t^{\op})$.

Let us define $\iota\colon\Mc(t)\simeq\Mc(t^{\op})$ as an isomorphism
on $\Z^{I}$ such that $\iota(f_{i}(t))=\iota(f_{i}(t^{\op}))$. Correspondingly,
by defining $\iota(X^{m})=X^{m}$, we obtain natural $\kk$-algebra
anti-isomorphisms
\begin{gather}
\iota \colon\ \LP(t)\simeq\LP(t^{\op}),\label{eq:opposite_LP}
\\
\iota \colon\ \hLP(t)\simeq\tLP(t^{\op}),\label{eq:opposite_hLP}
\\
\iota \colon\ \tLP(t)\simeq\hLP(t^{\op}).\nonumber
\end{gather}
Notice that $\iota\colon\LP(t)\simeq\LP(t^{\op})$ induces an anti-isomorphism
$\iota\colon\cF(t)\simeq\cF(t^{\op})$.

For any given $k\in I_{\ufv}$, we have $\mu_{k}(t^{\op})=(\mu_{k}t)^{\op}$.
It is straightforward to check the commutativity of the following
diagram:
\begin{gather}
\begin{array}{ccc}
\cF(t) & \xrightarrow{\iota} & \cF(t^{\op})\\
\uparrow\mu_{k}^{*} & & \uparrow\mu_{k}^{*}\\
\cF(\mu_{k}t) & \xrightarrow{\iota} & \cF(\mu_{k}(t^{\op})).
\end{array}\label{eq:comm_mutation_op}
\end{gather}
In particular, $\iota(\mu_{k}^{*}X_{i}(\mu_{k}t))=\mu_{k}^{*}(\iota X_{i}(\mu_{k}t))$
is given by $X_{i}(t^{\op})$ if $i\neq k$, or
\[
X(t^{\op})^{-f_{k}(t^{\op})+\sum_{j}[-b_{jk}]_{+}f_{j}(t^{\op})}+X(t^{\op})^{-f_{k}(t^{\op})+\sum_{i}[b_{ik}]_{+}f_{i}(t^{\op})}
\]
 if $i=k$.

Notice that $Y(t)^{n}=X^{\tB n}$ while $Y(t^{\op})^{n}=X^{-\tB n}$.
It follows that $Z\in\hLP(t)$ is $g$-pointed if and only if $\iota Z\in\tLP(t^{\op})$
is $g$-copointed. We have the following result.

\begin{Lem}
Let there be given seeds $t'=\seq_{t',t}t$. Then the codegree tropical
transformation $\coTrop_{t',t}\colon\Mc(t)\simeq\Mc(t')$ equals the composition
$\Mc(t)\xrightarrow{\iota}\Mc(t^{\op})\xrightarrow{\phi_{(t')^{\op},t^{\op}}}\Mc((t')^{\op})\xrightarrow{\iota}\Mc(t')$.
In particular, it is independent of the choice of $\seq_{t',t}$.
\end{Lem}

\begin{proof}
By the commutativity between $\iota$ and mutations, it suffices to
check the claim for adjacent seeds $t'=\mu_{k}t$, which follows from
definition.
\end{proof}

Notice that we have $\LP(t;t')=\iota\LP(t^{\op};(t')^{\op})$ and
$\qUpClAlg(t)=\iota(\qUpClAlg(t^{\op}))$ by the commutativity between
$\iota$ and mutations.

\begin{Def}[codegree compatibility]
If $Z$ belongs to $\LP(t;t')\subset\LP(t)$, then $Z$ is said to
be compatibly copointed at $t$, $t'$ if it is $\eta$-copointed for
some $\eta\in\Mc(t)$, and $\seq_{t,t'}^{*}Z$ is $\coTrop_{t',t}\eta$-copointed.

If $Z$ belongs to $\qUpClAlg(t)\subset\LP(t)$, then $Z$ is said
to be compatibly copointed at $\Delta^{+}$ if it is compatibly copointed
at $t$, $t'$ for any $t'\in\Delta^{+}$.

Let $S$ denote a set consisting of $\eta$-copointed elements $S^{\eta}\in\tLP(t)$
for distinct $\eta\in\Mc(t)$. If~$S^{\eta}$ are compatibly copointed
at $t$, $t'$ for all $\eta$, we say $S$ is compatibly copointed at
$t$, $t'$, or the copointed sets $S$ and $\seq_{t,t'}^{*}S$ are (codegree)
compatible.
\end{Def}

\begin{Rem}
We refer the reader to \cite[Section 3.5]{kashiwara2019laurent} for
a categorical view. In their setting, to study cluster algebras arising
from quantum unipotent cells, one considers certain module category
of the corresponding quiver Hecke algebras and localize it at the
simple objects corresponding to the frozen variables. Then the degrees
and codegrees, together with the tropical transformations, can be
calculated by taking dual objects in the localized module category.
\end{Rem}

\subsection{Injective-reachability and distinguished functions\label{subsec:Injective-reachability}}

Let $\sigma$ denote a permutation of $I_{\ufv}$. For any mutation
sequence $\seq=\mu_{k_{r}}\cdots\mu_{k_{1}}$, we define $\sigma\seq=\mu_{\sigma k_{r}}\cdots\mu_{\sigma k_{1}}$.

Let $\pr_{I_{\ufv}}$ and $\pr_{I_{\fv}}$ denote the natural projection
from $\Z^{I}$ to $\Z^{I_{\ufv}}$ and $\Z^{I_{\fv}}$ respectively.

\begin{Def}[{\cite[Definition 5.1.1]{qin2017triangular}}]\label{def:inj_reachable}
A seed $t$ is said to be injective-reachable if there exists a mutation
sequence $\seq=\seq_{t',t}$ and a permutation $\sigma$ of $I_{\ufv}$,
such that the seed $t'=\seq_{t',t}t$ satisfies $b_{\sigma i,\sigma j}(t')=b_{ij}(t)$
for $i,j\in I_{\ufv}$ and, for any $k\in I_{\ufv}$,
\begin{gather}
\deg^{t}\seq_{t',t}^{*}X_{\sigma k}(t') = -f_{k}+u_{k}\label{eq:deg_inj}
\end{gather}
for some $u_{k}\in\Z^{I_{\fv}}$.

In this case, we denote $t'=t[1]$ and say it is shifted from $t$
(by $[1]$) with the permutation~$\sigma$. Similarly, we denote $t=t'[-1]$
and say it is the shifted from $t'$ (by $[-1]$) with the permutation~$\sigma^{-1}$.
\end{Def}

We recall that, up to a permutation of vertices, a seed is determined
by the degrees (exten\-ded $g$-vectors) of its cluster variables, see
\cite{gross2018canonical} for an interpretation in terms of chambers.
In~parti\-cu\-lar, the shifted seed $t[1]$ is unique up to a permutation.

Let there be given an injective-reachable seed $t$. Recursively,
we construct a chain of seeds $\{t[d]\,|\, d\in\Z\}$ called an injective-reachable
chain, such that $t[d]=\big(\sigma^{d-1}\seq\big)t[d-1]$, see \cite[Definition~5.2.1]{qin2017triangular}.
In particular, we have $t=\big(\sigma^{-1}\seq\big)t[-1]$.

We denote $I_{k}(t)=\seq_{t[1],t}^{*}X_{\sigma k}(t[1])$ and $P_{k}(t)=\seq_{t[-1],t}^{*}X_{\sigma^{-1}(k)}(t[-1])$.
For any $d\in\N^{I_{\ufv}}$, define the cluster monomial $I(t)^{d}:=\big[\prod_{k}I_{k}(t)^{d_{k}}\big]^{t}$
and $P(t)^{d}:=\big[\prod_{k}P_{k}(t)^{d_{k}}\big]^{t}$.

Since a quantum cluster monomial is pointed, it is also copointed
by \cite{FominZelevinsky07} (we can also see this using the map $\iota$).
It follows that $I(t)^{d}=\big\{\prod_{k}I_{k}(t)^{d_{k}}\big\}^{t}$ and
$P(t)^{d}=\big\{\prod_{k}P_{k}(t)^{d_{k}}\big\}^{t}$.

Notice that if $t$ is injective-reachable, then so is any seed $t'\in\Delta^{+}$.
Such property is equivalent to the existence of a green to red sequence.
See \cite{qin2017triangular,qin2019bases} for more details.

For any $g=(g_{i})_{i\in I}\in\Z^{I}\simeq\Mc(t)$, denote $[g]_{+}=([g_{i}]_{+})_{i\in I}$.
We have the following $g$-pointed element in $\LP(t)$:
\begin{gather*}
\Inj_{g}^{t} = \big[p_{g}*X(t)^{[g]_{+}}*I(t)^{[-\pr_{I_{\ufv}}g]_{+}}\big]^{t}
\end{gather*}
for some frozen factor $p_{g}\in\cRing$. Define the following set
of distinguished pointed functions
\begin{gather*}
\Inj^{t} := \big\{\Inj_{g}^{t}\,|\, g\in\Mc(t)\big\}.
\end{gather*}

Denote $t'=t[1]$. By~\eqref{eq:deg_inj}, the linear map $\psi_{t,t'}:\Mc(t')\simeq\Mc(t)$
is determined by
\begin{gather*}\begin{split}
&\psi_{t,t'}(f_{\sigma k}') = -f_{k}+u_{k},\qquad u_{k}\in\Z^{I_{\fv}},\quad k\in I_{\ufv},
\\
&\psi_{t,t'}(f_{i}') = f_{i},\qquad i\in I_{\fv}.\end{split}
\end{gather*}
Using Proposition~\ref{prop:swap_deg_codeg}, we deduce that
\begin{gather}
\codeg^{t[1]}P_{\sigma k}(t[1]) = \codeg^{t[1]}\seq_{t,t[1]}^{*}X_{k}(t)=\psi_{t,t[1]}^{-1}f_{k}
 = -f_{\sigma k}(t[1])+u_{k}. \label{eq:codeg_proj_shift}
\end{gather}
Notice that~\eqref{eq:codeg_proj_shift} appears in \cite[equation~(18)]{qin2017triangular}
as an assumption. Replacing $t$ by $t[-1]$ in the above argument,
we obtain
\begin{gather}
\codeg^{t}P_{\sigma k}(t) = -f_{\sigma k}+u_{k}'\label{eq:codeg_proj}
\end{gather}
for any $k\in I_{\ufv}$ and some $u_{k}'\in\Z^{I_{\fv}}$.

Correspondingly, for any $\eta\in\Z^{I}\simeq\Mc(t)$, we have the
following $\eta$-copointed element in~$\LP(t)$:
\begin{gather*}
\Proj^{t,\eta} = \big\{P(t)^{[-\pr_{I_{\ufv}}\eta]_{+}}*X(t)^{[\eta]_{+}}*p^{\eta}\big\}^{t}
\end{gather*}
for some frozen factor $p^{\eta}\in\cRing$. Define the following
set of distinguished copointed functions
\begin{gather*}
\Proj^{t} := \big\{\Proj^{t,\eta}\,|\, \eta\in\Mc(t)\big\}.
\end{gather*}

The two kinds of distinguished functions are related by the following
result. At a categorical level, it can be viewed as the duality between
injective representations and projective representations for a pair
of opposite quivers, see \cite[Section 5.3]{qin2017triangular} for
more discussion.

\begin{Lem}
Denote $\seq=\seq_{t[1],t}$. The following claims are true.
\begin{enumerate}\itemsep=0pt
\item[$(1)$] For any $k\in I_{\ufv}$, we have $\iota P_{k}(t)=I_{k}(t^{\op})$.

\item[$(2)$] We have $t[-1]^{\op}=(t^{\op})[1]=\big(\sigma^{-1}\seq\big)^{-1}t^{\op}$,
which is shifted from $t^{\op}$ with the permutation~$\sigma^{-1}$.

\item[$(3)$] We have $t[1]^{\op}=(t^{\op})[-1]=\seq t^{\op}$, which is shifted
from $t^{\op}$ with the permutation $\sigma$.
\end{enumerate}
\end{Lem}

\begin{proof}$(1)$ Recall that $\iota P_{k}(t)$ is a quantum cluster variable contained
in $\LP(t^{\op})$. By~\eqref{eq:codeg_proj}, $\iota P_{k}(t)$ is
pointed at $-f_{k}+u$ for some $u\in\Z^{I_{\fv}}$. The claim follows.

$(2)$ Denote $\seqnu=\big(\sigma^{-1}\seq\big)^{-1}$ for simplicity. By the
commutativity between mutations and $\iota$, we have $t[-1]^{\op}=(\seqnu t)^{\op}=\seqnu t^{\op}$.

The seed $t[-1]^{\op}$ has the principal $B$-matrix given by
$b_{ij}(t[-1]^{\op})=-b_{ij}(t[-1])=-b_{\sigma i,\sigma j}$,
$i,j\in I_{\ufv}$. Using the commutativity relation~\eqref{eq:comm_mutation_op}
between $\iota\colon \cF(t)\simeq\cF(t^{\op})$, $\iota\colon\cF(\seqnu t)\simeq\cF((\seqnu t)^{\op})$
and mutations, its cluster variables have the following Laurent expansion
in $\LP(t^{\op})$:
\begin{gather*}
\big(\seqnu\big)^{*}\big(X_{\sigma^{-1}k}\big(\seqnu t^{\op}\big)\big)
=\iota\big(\big(\seqnu\big)^{*}\iota^{-1}X_{\sigma^{-1}k}\big(\seqnu t^{\op}\big)\big)
=\iota\big(\big(\seqnu\big)^{*}X_{\sigma^{-1}k}\big(\seqnu t\big)\big)=\iota\big(P_{k}(t)\big),
\end{gather*}
$k\in I_{\ufv}$ , which are pointed at $-f_{k}+u$, $u\in\Z^{I_{\fv}}$.
It follows that $t[-1]^{\op}$ is a shifted seed $t^{\op}[1]$ with
the permutation $\sigma^{-1}$.

(3) Notice that $t[1]^{\op}=(\seq t)^{\op}=\seq(t^{\op})$ by the
commutativity between mutations and $\iota$. Since $\seq^{-1}t^{\op}=t^{\op}[1]$
with the permutation $\sigma^{-1}$ by (2), we have $\seq t^{\op}=t^{\op}[-1]$
with the permutation $\sigma$.
\end{proof}

\begin{Eg}
Let us continue Example~\ref{eg:A2_cluster}. Ignore the Lambda matrices
for simplicity and set $v=1$. Then the cluster algebra have five
cluster variables in $\LP(t_{0})$:
\begin{gather*}
X_{1},\qquad
X_{2},\qquad
P_{2} = \frac{X_{1}+1}{X_{2}},\qquad
P_{1}=I_{2}=\frac{X_{1}+1+X_{2}}{X_{1}X_{2}},\qquad
I_{1}=\frac{1+X_{2}}{X_{1}}.
\end{gather*}

Take the seed $t_{0}=(B,(X_{1},X_{2}))$. Choose the mutation sequence
$\seq=\mu_{2}\mu_{1}\mu_{2}$ (read from right to left). Then
\begin{gather*}
t_{0}[1] : =\seq t_{0}=(-B,(I_{2},I_{1}))
\end{gather*}
is shifted from $t_{0}$ by $[1]$ with the unique non-trivial permutation
$\sigma$ of $I=\{1,2\}$. Notice that $\big(\sigma^{-1}\seq\big)^{-1}=\mu_{1}\mu_{2}\mu_{1}$.
Similarly,
\begin{gather*}
t_{0}[-1] : =\big(\sigma^{-1}\seq\big)^{-1}t_{0}=(-B,(P_{2},P_{1}))
\end{gather*}
is shifted from $t_{0}$ by $[-1]$ with $\sigma^{-1}$.

It follows that
\begin{gather*}
t_{0}^{\op} = (-B,(\iota X_{1},\iota X_{2}))=\big({-}B,(X_{1}(t_{0}^{\op}),X_{2}(t_{0}^{\op}))\big),
\\
t_{0}[1]^{\op} = (B,(\iota I_{2},\iota I_{1}))=\big(B,(P_{2}(t_{0}^{\op}),P_{1}(t_{0}^{\op}))\big)=\seq t_{0}^{\op},
\\
t_{0}[-1]^{\op} = (B,(\iota P_{2},\iota P_{1}))=\big(B,(I_{2}(t_{0}^{\op}),I_{1}(t_{0}^{\op}))\big)=\big(\sigma^{-1}\seq\big)^{-1}t_{0}^{\op}.
\end{gather*}
We see that $t_{0}[-1]^{\op}$ is shifted from $t_{0}^{\op}$ by $[1]$
with the permutation $\sigma^{-1}$, which we denote by $t_{0}[-1]^{\op}=t_{0}^{\op}[1]$,
and $t_{0}[1]^{\op}$ is shifted from $t_{0}^{\op}$ by $[-1]$ with
the permutation $\sigma$, which we denote by $t_{0}[1]^{\op}=t_{0}^{\op}[-1]$.
\end{Eg}

\begin{Lem}[substitution]\label{lem:substitution}\quad
\begin{enumerate}\itemsep=0pt
\item[$(1)$] Assume that $\big[X(t)^{d}*I(t)^{d'}\big]^{t}$
is $(\prec_{t},\mm)$-unitriangular to $\Inj^{t}$ for any $d\in\N^{I_{\ufv}}\oplus\Z^{I_{\fv}}$
and \mbox{$d'\in\N^{I_{\ufv}}$}. If $Z$ is $(\prec_{t},\mm)$-unitriangular
to $\Inj^{t}$, then the normalized products $\big[X(t)^{d}*Z*I(t)^{d'}\big]^{t}$
are $(\prec_{t},\mm)$-unitriangular to $\Inj^{t}$ too {\rm\cite[Lemma~6.2.4]{qin2017triangular}}.

\item[$(2)$] Assume that $\big\{P(t)^{d'}*X(t)^{d}\big\}^{t}$ is codegree $(\succ_{t},\mm)$-unitriangular
to $\Proj^{t}$ for any $d\in\N^{I_{\ufv}}\oplus\Z^{I_{\fv}}$ and
$d'\in\N^{I_{\ufv}}$. If $Z$ is codegree $(\succ_{t},\mm)$-unitriangular
to $\Proj^{t}$, then the codegree normalized products $\big\{P(t)^{d'}*Z*X(t)^{d}\big\}^{t}$
are codegree $(\succ_{t},\mm)$-unitriangular to $\Proj^{t}$ too.
\end{enumerate}
\end{Lem}

\begin{proof}
(1) has been proved in \cite{qin2017triangular}. We can prove (2)
using similar arguments as those for (1), or deduce (2) from (1) by
using the map $\iota$.
\end{proof}

We have the following relation between degree and codegree tropical
transformations.

\begin{Prop}\label{prop:deg_codeg_trop_trans}

For any $t,t'\in\Delta^{+}$, the following diagram commutes:
\begin{gather}
\begin{array}{ccc}
\Mc(t[1]) & \xrightarrow{\psi_{t,t[1]}} & \Mc(t)\\
\downarrow\coTrop_{t'[1],t[1]} & & \downarrow\phi_{t',t}\\
\Mc(t'[1]) & \xrightarrow{\psi_{t',t'[1]}} & \Mc(t').
\end{array}
\end{gather}
\end{Prop}

\begin{proof}
It suffices to check the claim for the case $t'=\mu_{k}t$, $k\in I_{\ufv}$.
Notice that, in this case, we~have $t'[1]=\mu_{\sigma k}(t[1])$ and
$\seq_{t,t'}^{*}I_{i}(t)=I_{i}(t')$ for $i\neq k$, see \cite[Proposition 5.1.4]{qin2017triangular}.

Notice that, all $u\in\Z^{I_{\fv}}$ are invariant under the maps
in the diagrams. In view of the piecewise linearity of $\phi_{t',t}$
and $\coTrop_{t'[1],t[1]}$, it remains to check the claim that, for
$i\in I_{\ufv}$,
\begin{gather*}
\phi_{t',t}\psi_{t,t[1]}(\pm f_{\sigma i}(t[1]))
= \psi_{t',t'[1]}\coTrop_{t'[1],t[1]}(\pm f_{\sigma i}(t[1])).
\end{gather*}

(i) By definition, for $i\neq k$ in $I_{\ufv}$, we have
\begin{gather*}
\phi_{t',t}\psi_{t,t[1]}(f_{\sigma i}(t[1])) =\deg^{t'}\seq_{t,t'}^{*}I_{i}(t)=\deg^{t'}I_{i}(t')
\end{gather*}
and also
\begin{gather*}
\psi_{t',t'[1]}\coTrop_{t'[1],t[1]}(f_{\sigma i}(t[1])) =\text{\ensuremath{\psi_{t',t'[1]}(f_{\sigma i}(t'[1])}})=\deg^{t'}I_{i}(t').
\end{gather*}
It follows that these two vectors in $\Mc(t')$ agree.

(ii) For the non-trivial case $i=k$, we have
\begin{gather*}
\phi_{t',t}\psi_{t,t[1]}(f_{\sigma k}(t[1])) =\deg^{t'}\seq_{t,t'}^{*}I_{k}(t).
\end{gather*}
Notice that $I_{k}(t)$ and $I_{k}(t')$ are related by an exchange
relation for the seeds $(t[1],t'[1])$. It~follows that we have
\begin{gather*}
\deg^{t'}\seq_{t,t'}^{*}I_{k}(t) = -\deg^{t'}I_{k}(t')+\sum_{i\in I_{\ufv}}[b_{ik}(t')]_{+}\deg^{t'}I_{i}(t')
+\sum_{s\in I_{\fv}}[b_{s,\sigma k}(t'[1])]_{+}f_{s},
\end{gather*}
see \cite[equation~(14)]{qin2017triangular}.

On the other hand, we have the following computation
\begin{gather*}
\psi_{t',t'[1]}\coTrop_{t'[1],t[1]}(f_{\sigma k}(t[1]))
\\ \qquad
{}= \psi_{t',t'[1]}\bigg({-}f_{\sigma k}(t'[1])+\sum_{i\in I_{\ufv}}[-b_{\sigma i,\sigma k}(t[1])]_{+}f_{\sigma i}(t'[1])+\sum_{s\in I_{\fv}}[-b_{s,\sigma k}(t[1])]_{+}f_{s}\bigg)
\\ \qquad
{}= -\deg^{t'}I_{k}(t')+\sum_{i\in I_{\ufv}}[-b_{\sigma i,\sigma k}(t[1])]_{+}\deg^{t'}I_{i}(t')+\sum_{s\in I_{\fv}}[-b_{s,\sigma k}(t[1])]_{+}f_{s}
\\ \qquad
{}= -\deg^{t'}I_{k}(t')+\sum_{i\in I_{\ufv}}[b_{ik}(t')]_{+}\deg^{t'}I_{i}(t')+\sum_{s\in I_{\fv}}[b_{s,\sigma k}(t'[1])]_{+}f_{s}.
\end{gather*}

The desired equality follows:
\begin{gather*}
\phi_{t',t}\psi_{t,t[1]}(f_{\sigma k}(t[1])) = \psi_{t',t'[1]}\coTrop_{t'[1],t[1]}(f_{\sigma k}(t[1])).
\end{gather*}

(iii) By~\eqref{eq:deg_inj} and the linearity of $\psi_{t,t[1]}$,
$\psi_{t',t'[1]}$, for $i\neq k$ in $I_{\ufv}$, we have
\begin{gather*}
\phi_{t',t}\psi_{t,t[1]}(-f_{\sigma i}(t[1])) =\phi_{t',t}(f_{i}(t)-u_{i})=f_{i}(t')-u_{i}
\end{gather*}
and also
\begin{gather*}
\psi_{t',t'[1]}\coTrop_{t'[1],t[1]}(-f_{\sigma i}(t[1])) =\text{\ensuremath{\psi_{t',t'[1]}(-f_{\sigma i}(t'[1])}}) =-\deg^{t'}I_{i}(t').
\end{gather*}
Then~\eqref{eq:deg_inj} implies that the two vectors in $\Mc(t')$
agree.

(iv) For the non-trivial case $i=k$, the linearity of $\psi_{t,t[1]}$
implies
\begin{gather*}
\psi_{t,t[1]}(-f_{\sigma k}(t[1])) =-\deg^{t}I_{k}(t).
\end{gather*}
Notice that $I_{k}(t)$ and $I_{k}(t')$ are related by an exchange
relation for the seeds $(t[1],t'[1])$. It~follows that we have
\begin{gather*}
\deg^{t}\seq_{t',t}^{*}I_{k}(t') = -\deg^{t}I_{k}(t)+\sum_{i\in I_{\ufv}}[b_{ik}(t)]_{+}\deg^{t}I_{i}(t)
 +\sum_{s\in I_{\fv}}[b_{s,\sigma k}(t[1])]_{+}f_{s},
\end{gather*}
see \cite[equation~(14)]{qin2017triangular}. Consequently, we get
\begin{gather*}
\psi_{t,t[1]}(-f_{\sigma k}(t[1])) =\deg^{t}\seq_{t',t}^{*}I_{k}(t')-\sum_{i\in I_{\ufv}}[b_{ik}(t)]_{+}\deg^{t}I_{i}(t)
 -\sum_{s\in I_{\fv}}[b_{s,\sigma k}(t[1])]_{+}f_{s}.
\end{gather*}

On the other hand, we compute that
\begin{gather*}
\psi_{t',t'[1]}\coTrop_{t'[1],t[1]}(-f_{\sigma k}(t[1]))
\\ \qquad
{}=\text{\ensuremath{\psi_{t',t'[1]}(f_{\sigma k}(t'[1])}}-\sum_{i\in I_{\ufv}}[b_{\sigma i,\sigma k}(t[1])]_{+}f_{\sigma i}(t'[1])-\sum_{s\in I_{\fv}}[b_{s,\sigma k}(t[1])]_{+}f_{s})
\\ \qquad
{}=\deg^{t'}I_{k}(t')-\sum_{i\in I_{\ufv}}[b_{\sigma i,\sigma k}(t[1])]_{+}\deg^{t'}I_{i}(t')-\sum_{s\in I_{\fv}}[b_{s,\sigma k}(t[1])]_{+}f_{s}
\\ \qquad
{}=\deg^{t'}I_{k}(t')-\sum_{i\in I_{\ufv}}[b_{ik}(t)]_{+}\deg^{t'}I_{i}(t')-\sum_{s\in I_{\fv}}[b_{s,\sigma k}(t[1])]_{+}f_{s}.
\end{gather*}
Further applying $\phi_{t,t'}$ to both sides, we obtain
\begin{gather*}
\phi_{t,t'}\psi_{t',t'[1]}\coTrop_{t'[1],t[1]}(-f_{\sigma k}(t[1]))
 \\ \qquad{}
 {}=\deg^{t}\seq_{t',t}^{*}I_{k}(t')-\sum_{i\in I_{\ufv}}[b_{ik}(t)]_{+}\deg^{t}I_{i}(t)-\sum_{s\in I_{\fv}}[b_{s,\sigma k}(t[1])]_{+}f_{s}.
\end{gather*}

Consequently, we have
\begin{gather*}
\psi_{t,t[1]}(-f_{\sigma k}(t[1])) = \phi_{t,t'}\psi_{t',t'[1]}\coTrop_{t'[1],t[1]}(-f_{\sigma k}(t[1])).
\end{gather*}
We obtain the claim observing that $\phi_{t,t'}=\phi_{t',t}^{-1}$.
\end{proof}

Consequently, we obtain a relation between the degree compatibility
and the codegree compatibility, which will be useful for studying
properties of double triangular bases (Proposition~\ref{prop:common_tri_basis_opposite}).

\begin{Prop}\label{prop:compatibily_pointed_copointed}
Let there be given seeds $t,t'\!\in\!\Delta^{+}$ and $Z\!\in\!\LP(t)\cap\allowbreak \seq_{t',t}^{*}\LP(t')\cap \allowbreak \seq_{t[1],t}^{*}\LP(t[1])\allowbreak \cap\seq_{t'[1],t}^{*}\LP(t'[1])$.
Then $Z$ is compatibly copointed at $t[1]$, $t'[1]$ if and only if
it is compatibly pointed at $t$, $t'$.
\end{Prop}

\begin{proof}
By Proposition~\ref{prop:swap_deg_codeg}, $\seq_{t,t[1]}^{*}Z$ is
$\eta$-copointed in $\LP(t[1])$ if and only if $Z$ is $\psi_{t,t[1]}\eta$-pointed
in $\LP(t)$, and similar statements hold in $\LP(t'[1])$ and $\LP(t')$.
Let us explain how the claim follows from Proposition~\ref{prop:deg_codeg_trop_trans}.

First, assume that $Z$ is compatibly copointed at $t[1]$, $t'[1]$.
This means that there exists some $\eta$ such that $\seq_{t,t[1]}^{*}Z$
is $\eta$-copointed in $\LP(t[1])$ and $\seq_{t,t'[1]}^{*}Z$ is
$\phi_{t'[1],t[1]}^{\op}\eta$-copointed in~$\LP(t'[1])$. Then Proposition~\ref{prop:swap_deg_codeg} implies that $Z$ is $\psi_{t,t[1]}\eta$-pointed
in $\LP(t)$ and $\seq_{t,t'}^{*}Z=\seq_{t'[1],t'}^{*}\big(\seq_{t,t'[1]}^{*}Z\big)$
is $\psi_{t',t'[1]}\big(\phi_{t'[1],t[1]}^{\op}\eta\big)$-pointed in~$\LP(t')$.
In addition, Proposition~\ref{prop:deg_codeg_trop_trans} implies
that $\psi_{t',t'[1]}\big(\phi_{t'[1],t[1]}^{\op}\eta\big)=\phi_{t',t}(\psi_{t,t[1]}\eta)$.
Therefore, $Z$ is compatibly pointed at $t$, $t'$.

Conversely, assume that $Z$ is compatibly pointed at $t$, $t'$. This
means that there exists some~$g$ such that $Z$ is $g$-pointed in
$\LP(t)$ and $\seq_{t,t'}^{*}Z$ is $\phi_{t',t}g$-pointed in $\LP(t')$.
Then Proposition~\ref{prop:swap_deg_codeg} implies that $\seq_{t,t[1]}^{*}Z$
is $\psi_{t,t[1]}^{-1}g$-copointed in $\LP(t[1])$ and $\seq_{t,t'[1]}^{*}Z$
is $\psi_{t',t'[1]}^{-1}(\phi_{t',t}g)$-copointed in~$\LP(t'[1])$.
In addition, Proposition~\ref{prop:deg_codeg_trop_trans} implies
that $\phi_{t'[1],t[1]}^{\op}\big(\psi_{t,t[1]}^{-1}g\big)=\psi_{t',t'[1]}^{-1}(\phi_{t',t}g)$.
Therefore, $Z$ is compatibly copointed at $t[1]$, $t'[1]$.
\end{proof}

\section{Bidegrees and bases\label{sec:Bidegrees-and-bases}}

Let there be given an injective-reachable quantum seed $t$ and a
subalgebra $\midAlg(t)\subset\qUpClAlg(t)$. \mbox{Assume} that $\midAlg(t)$
possesses a $\kk$-basis $\can$. Then $\midAlg(t)$ naturally gives
rise to a subalgebra $\midAlg(t'):=\seq_{t,t'}^{*}\midAlg(t)\subset\qUpClAlg(t')=\seq_{t,t'}^{*}\qUpClAlg(t)$,
and $\can$ naturally gives rise to a basis $\seq_{t,t'}^{*}\can$
of $\midAlg(t')$. We~some\-times omit the symbols $t$, $t'$, identifying
$\midAlg(t)$ and $\midAlg(t')$, $\can$ and $\seq_{t,t'}^{*}\can$.

\subsection{Bases with different properties\label{subsec:Bases-with-different}}

\begin{Def}[degree-triangular basis]
A $\kk$-basis $\can$ of $\midAlg(t)$ is said to be a degree-triangular
basis with respect to $t$ if the following conditions hold:
\begin{enumerate}\itemsep=0pt
\item[(1)] $X_{i}(t)\in\can$ for $i\in I$.

\item[(2)] \textit{Bar-invariance:} $\can$ is invariant under the bar involution.

\item[(3)] \textit{Degree parametrization:} $\can$ is $\Mc(t)$-pointed, i.e.,
it takes the form $\can=\{\can_{g}\,|\, g\in\Mc(t)\}$ such that $\can_{g}$
is $g$-pointed.

\item[(4)] \textit{Degree triangularity:} For any basis element $\can_{g}$, $i\in I$,
the decomposition of the pointed function $[X_{i}(t)*\can_{g}]^{t}$
in terms of $\can$ is degree $(\prec_{t},\mm)$-unitriangular:
\begin{gather*}
[X_{i}(t)*\can_{g}]^{t} = \sum_{g'\preceq_{t}g+f_{i}}b_{g'}\can_{g'},
\end{gather*}
where $b_{g+f_{i}}=1$, $b_{g'}\in\mm$ for $g'\prec_{t}g+f_{i}$.
\end{enumerate}

The basis is said to be a cluster degree-triangular basis with respect
to $t$, or a triangular basis for short, if it further contains the
quantum cluster monomials in $t$ and $t[1]$.
\end{Def}

It is not clear if a degree-triangular basis is unique or not. Nevertheless,
a triangular basis must be unique if it exists, see \cite[Lemma~6.3.2]{qin2017triangular}.
By definition, $\Inj^{t}$ is $(\prec_{t},\mm)$-unitriangular to
the triangular basis.

We now propose the dual version below.

\begin{Def}[codegree-triangular basis]\label{def:right_triangular_basis}
A $\kk$-basis $\can$ of $\midAlg(t)$ is said to be a codegree-triangular
basis with respect to $t$ if the following conditions hold:
\begin{enumerate}\itemsep=0pt
\item[(1)] $X_{i}(t)\in\can$ for $i\in I$.

\item[(2)] \textit{Bar-invariance:} $\can$ is invariant under the bar involution.

\item[(3)] \textit{Codegree parametrization:} $\can$ is $\Mc(t)$-copointed,
i.e., it takes the form $\can=\{\can^{\eta}\,|\,\eta\in\Mc(t)\}$ such
that $\can^{\eta}$ is $\eta$-copointed.

\item[(4)] \textit{Codegree triangularity:} For any basis element $\can^{\eta}$,
$i\in I$, the decomposition of the copointed function $\{\can^{\eta}*X_{i}(t)\}^{t}$
in terms of $\can$ is codegree $(\succ_{t},\mm)$-unitriangular:
\begin{gather*}
\{\can^{\eta}*X_{i}(t)\}^{t} = \sum_{\eta'\succeq_{t}\eta+f_{i}}c_{\eta'}\can^{\eta'},
\end{gather*}
where $c_{\eta+f_{i}}=1$, $c_{\eta'}\in\mm$ for $\eta'\succ_{t}\eta+f_{i}$.
\end{enumerate}

The basis is said to be a cluster codegree-triangular basis with respect
to $t$ if it further contains the quantum cluster monomials in $t$
and $t[-1]$.
\end{Def}

By definition, $\Proj^{t}$ is codegree $(\succ_{t},\mm)$-unitriangular
to the cluster codegree-triangular basis. Similar to \cite[Lemma~6.3.2]{qin2017triangular},
we can show that the cluster codegree-triangular basis is unique if
it exists.

\begin{Lem}[factorization]\label{lem:factorization}\quad
\begin{enumerate}\itemsep=0pt
\item[$(1)$] Let there be given a
degree-triangular basis $\can$. Then $[X_{i}(t)*S]^{t}=[S*X_{i}(t)]^{t}\in\can$
for any $i\in I_{\fv}$, $S\in\can$ {\rm \cite[Lemma~6.2.1]{qin2017triangular}}.

\item[$(2)$] Let there be given a codegree-triangular basis $\can$. Then $\{X_{i}(t)*S\}^{t}=\{S*X_{i}(t)\}^{t}\in\can$
for any $i\in I_{\fv}$, $S\in\can$.
\end{enumerate}
\end{Lem}

\begin{Def}[bidegree-triangular basis]
If $\can$ is both degree-triangular and codegree-tri\-an\-gu\-lar with
respect to $t$, we call it a bidegree-tri\-an\-gu\-lar basis with respect
to $t$.
\end{Def}

\begin{Def}[double triangular basis]\label{def:double_triangular_basis}
If $\can$ is bidegree-triangular with respect to $t$ and further
contains the quantum cluster monomials in $t$, $t[-1]$, $t[1]$, we call
it a cluster bidegree-triangular basis of $\midAlg(t)$ or a double
triangular basis with respect to $t$.
\end{Def}

The basis $\can$ in Example~\ref{eg:A2_cluster} provides an example
of double triangular bases.

\begin{Def}[common triangular basis]
Assume that $\can$ is the triangular basis of $\midAlg(t)$ with
respect to $t$. If, for any $t'\in\Delta^{+}$, $\seq_{t,t'}^{*}\can$
is the triangular basis of $\midAlg(t')=\seq_{t,t'}^{*}\midAlg(t)$
with respect to $t'$ and is compatible with $\can$, we call $\can$
the common triangular basis.
\end{Def}

\subsection{From triangular bases to double triangular bases}

\begin{Prop}\label{prop:triangular_to_double}
Let there be given the triangular basis $\can^{t}$ of $\midAlg(t)$
with respect to the seed~$t$. If~$\can^{t[-1]}:=\seq_{t,t[-1]}^{*}\can^{t}$
is the triangular basis with respect to~$t[-1]$, then $\can^{t}$
is the double triangular basis with respect to~$t$.
\end{Prop}

\begin{proof}
By assumption, $\can^{t}$ is the triangular basis for $t$ and $t[-1]$,
thus it must contain the quantum cluster monomials in $t$, $t[1]$
and $t[-1]$, $t$ respectively. It remains to check that $\can^{t}$
satisfies the defining conditions of a codegree triangular basis for
$t$, see Definition~\ref{def:right_triangular_basis}.

Conditions (1) and (2) are trivial as $\can^{t}$ is assumed to be
triangular for $t$. Since $\can^{t[-1]}$ is $\Mc(t[-1])$-pointed,
$\can^{t}=\seq_{t[-1],t}^{*}\can^{t[-1]}$ is $\Mc(t)$-copointed
by Proposition~\ref{prop:swap_deg_codeg}. Thus Condition~(3) is also
verified.

Now, let us prove Condition (4).

First, consider any $i\in I_{\fv}$. Then for any $V\in\can^{t}$
which is bipointed by (i), we have $\{V*X_{i}(t)\}^{t}=[V*X_{i}(t)]^{t}=[X_{i}(t)*V]^{t}\in\can^{t}$
by Lemma~\ref{lem:factorization}.

Second, consider any $k\in I_{\ufv}$ and any $\eta$-copointed element
$V\in\can^{t}$. Then $\seq_{t,t[-1]}^{*}X_{k}(t)=I_{\sigma^{-1}k}(t[-1])$,
and $\seq_{t,t[-1]}^{*}V$ is pointed at $g=\psi_{t[-1],t}\eta$.
Since $\seq_{t,t[-1]}^{*}V$ belongs to the triangular basis $\seq_{t,t[-1]}^{*}\can^{t}=\can^{t[-1]}$,
the normalized product
\begin{gather*} Z:=\big[\seq_{t,t[-1]}^{*}V*I_{\sigma^{-1}k}(t[-1])\big]^{t[-1]}=v^{\alpha}\seq_{t,t[-1]}^{*}V*I_{\sigma^{-1}k}(t[-1]),
\end{gather*}
$\alpha\in\Z$, is $(\prec_{t[-1]},\mm)$-unitriangular to $\Inj^{t[-1]}$
by Lemma~\ref{lem:substitution}. Therefore, it is $(\prec_{t[-1]},\mm)$-uni\-tri\-angular
to $\can^{t[-1]}$. Then it has the following finite $(\prec_{t[-1]},\mm)$-unitriangular
decomposition in~$\can^{t[-1]}$:
\begin{gather*}
Z = S^{(0)}+\sum_{j=1}^{r}b^{(j)}S^{(j)}
\end{gather*}
with $b^{(j)}\in\mm$, $r\in\N$, $\deg^{t[-1]}S^{(j)}\prec_{t[-1]}\deg^{t[-1]}S^{(0)}=\deg^{t[-1]}Z$
for $j>0$.

Applying the mutation $\seq_{t[-1],t}^{*}$, we obtain
\begin{gather*}
Z' : =\seq_{t[-1],t}^{*}Z=v^{\alpha}V*X_{k}(t)=\seq_{t[-1],t}^{*}S^{(0)}+\sum_{j=1}^{r}b^{(j)}\seq_{t[-1],t}^{*}S^{(j)}.
\end{gather*}
Proposition~\ref{prop:swap_deg_codeg} implies that $Z'$ is copointed
and, for any $j>0$, we must have\vspace{-.7ex}
\begin{gather*}
\codeg^{t}\seq_{t[-1],t}^{*}S^{(j)}\succ_{t}\codeg^{t}\seq_{t[-1],t}^{*}S^{(0)}=\codeg^{t}Z'.
\vspace{-.7ex}
\end{gather*}
Then this is a codegree $(\succ_{t},\mm)$-unitriangular decomposition
in~terms of the copointed set $\can^{t}$.
\end{proof}

We prove the following inverse result, although it will not be used
in this paper.

\begin{Prop}
Assume that $\can^{t}$ is the double triangular basis of $\midAlg(t)$
with respect to the seed~$t$. Then $\can^{t[-1]}:=\seq_{t,t[-1]}^{*}\can^{t}$
is the triangular basis with respect to~$t[-1]$.
\end{Prop}

\begin{proof}
By assumption, $\can^{t[-1]}$ contains the quantum cluster monomials
in $t[-1]$,~$t$. It remains to check that $\can^{t[-1]}$ satisfies
the definition condition of a degree triangular basis for $t[-1]$.

(i) Since $\can^{t}$ is $\Mc(t)$-copointed, $\can^{t[-1]}=\seq_{t,t[-1]}^{*}\can^{t}$
is $\Mc(t[-1])$-pointed by Proposition~\ref{prop:swap_deg_codeg}.

(ii-a) Take any $i\in I_{\fv}$. Then for any $(g-f_{i})$-pointed
element $V\in\can^{t[-1]}$, we have $X_{i}*V=v^{\alpha}X_{i}\cdot V=v^{2\alpha}V*X_{i}$
for some $\alpha\in\Z$. Since $X_{i}\cdot V$ is $g$-pointed, it
agrees with $[X_{i}*V]^{t[-1]}$. Moreover, $\seq_{t[-1],t}^{*}(X_{i}\cdot V)$
is $\eta$-copointed by Proposition~\ref{prop:swap_deg_codeg}, where
$\eta=\psi_{t[-1],t}^{-1}g$. Therefore, $\seq_{t[-1],t}^{*}(v^{-\alpha}X_{i}*V)=v^{-\alpha}X_{i}*\seq_{t[-1],t}^{*}V$
agrees with the copointed function $\big\{X_{i}*\seq_{t[-1],t}^{*}V\big\}^{t}$.
Using Lemma~\ref{lem:factorization}, we deduce that $\seq_{t[-1],t}^{*}[X_{i}*V]^{t[-1]}=\big\{X_{i}*\seq_{t[-1],t}^{*}V\big\}^{t}$
is contained in the codegree triangular basis $\can^{t}$. Consequently,
$[X_{i}*V]^{t[-1]}$ belongs to $\can^{t[-1]}$.

(ii-b) Take any $k\in I_{\ufv}$ and $g$-pointed element $V\in\can^{t[-1]}$.
Then $\seq_{t[-1],t}^{*}X_{k}(t[-1])=$ \mbox{$P_{\sigma k}(t)\in\can^{t}$},
and $\seq_{t[-1],t}^{*}V$ is copointed at $\eta=\psi_{t[-1],t}^{-1}g$.
The function $\seq_{t[-1],t}^{*}[X_{k}(t[-1])*V]^{t[-1]}$ is copointed
by Proposition~\ref{prop:swap_deg_codeg}, i.e., $\seq_{t[-1],t}^{*}[X_{k}(t[-1])*V]^{t[-1]}=\{P_{\sigma k}(t)*\seq_{t[-1],t}^{*}V\}^{t}$.
Since~$\can^{t}$ is a double triangular basis, $\seq_{t[-1],t}^{*}V$
is codegree $(\succ_{t},\mm)$-unitriangular to $\Proj^{t}$. Lemma~\ref{lem:substitution} implies that $\{P_{\sigma k}(t)*\seq_{t[-1],t}^{*}V\}^{t}$
is codegree $(\succ_{t},\mm)$-unitriangular to $\Proj^{t}$ and,
consequently, is codegree $(\succ_{t},\mm)$-unitriangular to $\can^{t}$.
We obtain a finite codegree $(\succ_{t},\mm)$-unitriangular decomposition\vspace{-1ex}
\begin{gather*}
Z:=\big\{P_{\sigma k}(t)*\seq_{t[-1],t}^{*}V\big\}^{t} = \sum_{j=0}^{r-1}b^{(j)}S^{(j)}+S^{(r)}\vspace{-1ex}
\end{gather*}
with $r\in\N$, $b^{(j)}\in\mm$, $\codeg S^{(j)}\succ_{t}\codeg^{t}S^{(r)}=\codeg^{t}Z$
for $j<r$.

Applying the mutation $\seq_{t,t[-1]}^{*}$, we obtain\vspace{-1ex}
\begin{gather*}
 Z' :=\seq_{t,t[-1]}^{*}Z=[X_{k}(t[-1])*V]^{t[-1]}
 =\sum_{j=1}^{r}b^{(j)}\seq_{t,t[-1]}^{*}S^{(j)}+\seq_{t,t[-1]}^{*}S^{(r)}.\vspace{-1ex}
\end{gather*}
Proposition~\ref{prop:swap_deg_codeg} implies that $Z'$ is pointed
and, for any $j<r$, we have $\deg^{t[-1]}\seq_{t,t[-1]}^{*}S^{(j)}\prec_{t[-1]}\deg^{t[-1]}\seq_{t,t[-1]}^{*}S^{(r)}=\deg^{t[-1]}Z'$.
Therefore, this decomposition becomes a degree $(\prec_{t[-1]},\mm)$-unitriangular
decomposition in $\can^{t}$.\vspace{-1ex}
\end{proof}

\subsection{Properties of common triangular bases}

Recall that we have the $\kk$-algebra anti-isomorphism $\iota\colon\LP(t)\simeq\LP(t^{\op})$
such that $\iota(X^{m})=X^{m}$, see~\eqref{eq:opposite_LP}. Define
the subalgebra $\midAlg(t^{\op})=\iota\midAlg(t)\subset\qUpClAlg(t^{\op})$.

\begin{Prop}\label{prop:common_tri_basis_opposite}
If $\midAlg(t)$ possesses the common triangular basis $\can\subset\LP(t)$,
then $\midAlg(t^{\op})$ possesses the common triangular basis $\iota\can\subset\LP(t^{\op})$.
\end{Prop}

\begin{proof}
Notice that $\iota$ sends (quantum) cluster monomials $\seq_{t',t}^{*}X(t')^{m}$
to (quantum) cluster monomials $\seq_{(t')^{\op},t^{\op}}^{*}X((t')^{\op})^{m}$,
$m\in\N^{I_{\ufv}}$, because it commutes with mutations. In particular,
it gives a bijection between the sets of cluster monomials.

For any seed $t'\in\Delta^{+}$, because the common triangular basis
$\can$ gives rise to the double triangular basis for $t'$ by Proposition
\ref{prop:triangular_to_double}, it gives rise to a codegree triangular
basis $\can^{t'}\subset\LP(t')$ for $t'$. Then $\iota\can^{t'}\subset\LP((t')^{\op})$
is a degree triangular basis containing all cluster monomials. Therefore,
$\iota\can^{t'}$ is the triangular basis with respect to $(t')^{\op}$.

Moreover, for any $t,t'\in\Delta^{+}$, because the elements of $\can$
are compatibly pointed at $t[-1]$, $t'[-1]$, the elements of $\can$
are compatibly copointed at $t$, $t'$ by Proposition~\ref{prop:compatibily_pointed_copointed}.
It follows that the elements of $\iota\can$ are compatibly pointed
at $t^{\op}$, $(t')^{\op}$.

Therefore, $\iota\can$ is the common triangular basis by definition.
\end{proof}

Recall that a common triangular basis is necessarily compatibly pointed
at $\Delta^{+}$. We have the following results.

\begin{Thm}\label{thm:common_tri_to_double}
\looseness=-1
Let there be a $\kk$-subalgebra $\midAlg(t)$ of the upper quantum
cluster algebra $\qUpClAlg(t)$. Assume that $\midAlg(t)$ possesses
the common triangular basis $\can$. Then the following statements
are true.
\begin{enumerate}\itemsep=0pt
\item[$(1)$] $\seq_{t,t'}^{*}\can$ is the double triangular basis of $\midAlg(t')=\seq_{t,t'}^{*}\midAlg(t)$
for any seed $t'\in\Delta^{+}$.

\item[$(2)$] $\can$ is compatibly copointed at $\Delta^{+}$.
\end{enumerate}
\end{Thm}

\begin{proof}(1) The claim follows from Proposition~\ref{prop:triangular_to_double}.
(2) By Proposition~\ref{prop:common_tri_basis_opposite}, $\iota\can$
is the common triangular basis of $\midAlg(t^{\op})$, which is necessarily
compatibly pointed at $(\Delta^{+})^{\op}$. Applying $\iota$ again,
we deduce that $\can=\iota(\iota\can)$ is compatibly copointed at~$\Delta^{+}$.
\end{proof}

\section{Main results\label{sec:Main-results}}

\subsection{An analog of Leclerc's conjecture}

Let there be given an injective-reachable seed $t$ and a $\kk$-subalgebra
$\midAlg(t)$ of the upper quantum cluster algebra $\qUpClAlg(t)$.

\begin{Lem}\label{lem:leclerc_conj_one_seed}
Assume that $\midAlg(t)$ possesses a bidegree-triangular basis $\can$.
Take any $i\in I$ and $\gamma\in\Mc(t)$. Denote the codegree of
the $\gamma$-pointed basis element $\can_{\gamma}$ by $\eta$. Then
we have either $X_{i}(t)*\can_{\gamma}\in v^{\Z}\can$ or
\begin{gather*}
X_{i}(t)*\can_{\gamma} = v^{s}S+\sum_{j}b_{j}L^{(j)}+v^{h}H
\end{gather*}
such that $s>h\in\Z$, $b_{j}\in v^{h+1}\Z[v]\cap v^{s-1}\Z\big[v^{-1}\big]$,
and $S,L^{(j)},H$ are finitely many distinct elements of $\can$
with
\begin{gather*}
\deg^{t}H,\deg^{t}L^{(j)} \prec_{t} \deg^{t}S=f_{i}+\gamma,
\\
\codeg^{t}S,\codeg^{t}L^{(j)} \succ_{t} \codeg^{t}H=f_{i}+\eta.
\end{gather*}
Moreover, we have $s=\lambda(f_{i},\gamma)$, $h=\lambda(f_{i},\eta)$.
\end{Lem}

\begin{proof}
Omit the symbol $t$ for simplicity.

Denote the codegree of $\can_{\gamma}$ by $\eta=\gamma+\tB n$, where
$n\in\Nufv^{\geq0}(t)\simeq\N^{I_{\ufv}}$. Then $X_{i}*\can_{\gamma}$
has degree $f_{i}+\gamma$ with coefficient $v^{s}:=v^{\lambda(f_{i},\gamma)}$,
codegree $f_{i}+\eta$ with coefficient $v^{h}:=\ v^{\lambda(f_{i},\eta)}$.
It follows that $h=s+\lambda\big(f_{i},\tB n\big)\leq s$, where $h=s$ if
and only if $n_{i}=0$.

Because $\can$ is a degree-triangular basis, we have a degree $(\prec_{t},\mm)$-unitriangular
decomposition with finitely many $S^{(0)},\ldots,S^{(r)}\in\can$:{\samepage
\begin{gather}
[X_{i}*\can_{\gamma}]^{t}=v^{-s}X_{i}*\can_{\gamma} = S^{(0)}+\sum_{j>0}b^{(j)}S^{(j)}\label{eq:decomposition_deg_triangular}
\end{gather}
such that $b^{(j)}\in\mm$, $\deg S^{(j)}\prec\deg S^{(0)}=f_{i}+\gamma$
for $j>0$.}

(i) Assume $n_{i}=0$, then $v^{-s}X_{i}*\can_{\gamma}$ is pointed
and bar-invariant. Because every basis elements~$S^{(j)}$ appearing
in~\eqref{eq:decomposition_deg_triangular} are bar-invariant and
$b^{(j)}\in\mm$, it follows that $v^{-s}X_{i}*\can_{\gamma}=S^{(0)}\in\can$.

(ii) Assume $n_{i}\neq0$. Then $h<s$. In addition, $v^{-s}X_{i}*\can_{\gamma}$
is pointed but not bar-invariant, because it has the Laurent monomial
$v^{h-s}X^{\eta+f_{i}}$ at the codegree.

Notice that $v^{-h}X_{i}*\can_{\gamma}$ is copointed. Multiplying
the decomposition~\eqref{eq:decomposition_deg_triangular} by $v^{s-h}$
and applying the bar involution, we get a decomposition of copointed
elements
\begin{gather*}
v^{h}\can_{\gamma}*X_{i} = v^{h-s}S^{(0)}+\sum_{j>0}v^{h-s}\cdot\overline{b^{(j)}}S^{(j)}.
\end{gather*}
Because $\can$ is a codegree-triangular basis and $v^{h}\can_{\gamma}*X_{i}$
is copo\-inted, the above decomposition must be codegree $(\succ_{t},\mm)$-unitriangular.
But $v^{h-s}S^{(0)}$ is not copo\-inted since $S^{(0)}\in\can$ is
copointed but $h<s$. Relabeling $S^{(j)}$, $j>0$, if necessary,
we assume $\codeg S^{(j)}\succ_{t}\codeg S^{(r)}$ for $j<r$. Then
the codegree term $X^{\eta+f_{i}}$ is contributed from $S^{(r)}$
and $S^{(r)}$ is copointed at~$\codeg(\can_{\gamma}*X_{i})=\eta+f_{i}$
with decomposition coefficient $1=v^{h-s}\overline{b^{(r)}}$. In
addition, the rema\-i\-ning terms $S^{(j)}$, $0<j<r$ must have coefficients
$v^{h-s}\cdot\overline{b^{(j)}}$ in $\mm$. It follows that $b_{j}:=b^{(j)}v^{s}$
belongs to $v^{h+1}\Z[v]$ for $0<j<r$. The claim follows by taking
$S=S^{(0)}$, $H=S^{(r)}$, $L^{(j)}=S^{(j)}$ for~$0<j<r$.
\end{proof}

\begin{Thm}\label{thm:weaker_analog_leclerc}
Let there be given a $\kk$-subalgebra $\midAlg(t)$ of the upper
quantum cluster algebra~$\qUpClAlg(t)$. Assume that it has the
common triangular basis $\can$. Then, for any $i\in I$, $V\in\can$,
and any localized quantum cluster monomial $R$, we have either $R*V\in v^{\Z}\can$
or
\begin{gather}
R*V = v^{s}S+\sum_{j}b_{j}L^{(j)}+v^{h}H\label{eq:weak_analog_leclerc}
\end{gather}
such that $s>h\in\Z$, $b_{j}\in v^{h+1}\Z[v]\cap v^{s-1}\Z\big[v^{-1}\big]$,
and $S$, $L^{(j)}$, $H$ are finitely many distinct elements of $\can$.
\end{Thm}

\begin{proof}
Since $\can$ is the common triangular basis, Theorem~\ref{thm:common_tri_to_double}(1)
implies that $\seq_{t,t'}^{*}\can$ is the double triangular basis
(and thus bidegree-triangular) of $\midAlg(t')=\seq_{t,t'}^{*}\midAlg(t)$
for any seed $t'\in\Delta^{+}$. We apply Lemma~\ref{lem:leclerc_conj_one_seed}
for localized quantum cluster monomials associated to $t'$.
\end{proof}

Theorem~\ref{thm:weaker_analog_leclerc} is a weaker form of the following
analog of Leclerc's conjecture.

\begin{Conj}\label{conj:analog_leclerc}
Assume that $\can$ is the common triangular basis. Assume that $R$
is a real basis element in $\can$ $($i.e., $R^{2}\in\can)$. Then, for
any $V\in\can$, we have either $R*V\in v^{\Z}\can$ or
\begin{gather*}
R*V = v^{s}S+\sum_{j}b_{j}L^{(j)}+v^{h}H
\end{gather*}
such that $s>h\in\Z$, $b_{j}\in v^{h+1}\Z[v]\cap v^{s-1}\Z\big[v^{-1}\big]$,
and $S,L^{(j)},H$ are finitely many distinct elements of $\can$.
\end{Conj}

Choose any $l\in\N$. Let $\cC_{l}$ denote a level-$l$ subcategory
of the monoidal category of the finite-dimensional modules of a quantum
affine algebra $\envAlg(\widehat{\frg})$ in the sense of~\cite{HernandezLeclerc09},
where $\frg$ is a Lie algebra of type $ADE$. Let $K_{t}(\cC_{l})$
denote its $t$-deformed Grothendieck ring, $t$ a quantum para\-meter.
By \cite[Theorem~8.4.3]{qin2017triangular}, $K_{t}(\cC_{l})$ is
a (partially compactified) quantum cluster algebra~$\bQClAlg$. Notice
that $K_{t}(\cC_{l})$ has a bar-invariant basis $\{[S]\}$, where
$S$ are simple modules. By~\cite{qin2017triangular}, $\{[S]\}$~becomes the common triangular basis of the corresponding quantum cluster
algebra $\qClAlg$ after localization at the frozen factors.

A simple module $R$ in $\cC_{l}$ is called real if $R\otimes R$
remains simple. Theorem~\ref{thm:weaker_analog_leclerc} implies the
following result.

\begin{Thm}\label{thm:analog-leclerc-quantum-aff}
Let $R$ be any real simple module in $\cC_{l}$ corresponding to
a cluster monomial. Then, for any simple modules $V\in\cC_{l}$, either
$R\otimes V$ is simple, or there exists finitely many distinct simple
modules $S$, $L^{(j)}$, $H$ in $\cC_{l}$ such that the following equation
holds in the deformed Grothendieck ring $K_{t}(\cC_{l})$:
\begin{gather*}
[R]*[V] = t^{s}[S]+\sum_{j}b_{j}\big[L^{(j)}\big]+t^{h}[H],
\end{gather*}
where $s>h\in\Hf\Z$, $b_{j}\in t^{h+\Hf}\Z\big[t^{\Hf}\big]\cap t^{s-\Hf}\Z\big[t^{-\Hf}\big]$.
\end{Thm}

Notice that we can replace $[S]$ by the $t$-analog of $q$-character
of $S$ and embed $K_{t}(\cC_{l})$ into the completion of a quantum
torus, see \cite{Hernandez02,Nakajima04,VaragnoloVasserot03}.
Correspondingly, Theorem~\ref{thm:analog-leclerc-quantum-aff} gives
an algebraic relation for such characters.

\begin{Rem}\label{rem:categorical_leclerc}
Assume that the quantum cluster algebra arises from a quantum unipotent
subgroup of symmetric Kac--Moody type, which possesses the dual canonical
basis corresponding to the set of self-dual simple modules of the
corresponding quiver Hecke algebra. In this case, up to $v$-power
rescaling, $S$ and $H$ correspond to the simple socle and simple
head of the convolution product $R\circ V$ respectively. See \cite[Section 4]{Kang2018}
for more details.

Similar picture appears in the category of finite-dimensional modules
of quantum affine algebras $\envAlg(\widehat{\frg})$ by \cite{kashiwara2020monoidal},
where the objects are not graded. In particular, \cite[Corollary 4.12]{kashiwara2020monoidal}
implies that, as a stronger version of Theorem~\ref{thm:analog-leclerc-quantum-aff},
Conjecture~\ref{conj:analog_leclerc} holds true for the quantum cluster
algebras $K_{t}(\cC_{l})$.
\end{Rem}

\subsection{Properties of dual canonical bases}

Let us consider the quantum unipotent subgroup $\qO[N_{-}(w)]$ of
symmetrizable Kac--Moody types in the sense of \cite{Kimura10,qin2020bases}.
It is isomorphic to a (partially compactified) quantum cluster algebra
after rescaling, see \cite{goodearl2020integral,GY13} or~\cite{qin2020bases}. Theorem~\ref{thm:weaker_analog_leclerc} implies the following weaker version of~Conjecture~\ref{conj:leclerc}.

\begin{Thm}\label{thm:weaker_leclerc}
Consider the dual canonical basis $\dCan(w)$ of $\qO[N_{-}(w)]$.
If $b_{1}\in\dCan(w)$ corresponds to a quantum cluster monomial after
rescaling, then for any $b_{2}\in\dCan(w)$, either $b_{1}b_{2}\in q^{\Z}\dCan(w)$
or~\eqref{eq:Leclerc_conj} holds true.
\end{Thm}

\begin{proof}
By \cite[Theorem~9.5.1]{qin2020bases}, after rescaling and localization
at the frozen factors, the dual canonical basis $\dCan(w)$ of $\qO[N_{-}(w)]$
becomes the common triangular basis of the corresponding quantum cluster
algebra. Therefore, elements of $\dCan(w)$ satisfy the algebraic
relation~\eqref{eq:weak_analog_leclerc} after rescaling. Notice that
the rescaling factors depends on the natural root-lattice grading
of $\envAlg$, which is homogeneous for $\seq_{t',t}^{*}X_{i}(t')*V$, $S$, $L^{(j)}$, $H$
in~\eqref{eq:weak_analog_leclerc}, because the $Y$-variables have
$0$-grading \cite[Section~9.1]{qin2020bases}. The claim follows
from Theorem~\ref{thm:weaker_analog_leclerc}.
\end{proof}

Theorem~\ref{thm:weaker_leclerc} would imply Conjecture~\ref{conj:leclerc}
if the following multiplicative reachability conjecture can be proved.

\begin{Conj}\label{conj:d_can_basis_reachability}
If $b\in\dCan(w)\subset\qO[N_{-}(w)]$ is real $($i.e., $b^{2}\in q^{\Z}\dCan(w))$,
then it corresponds to a quantum cluster monomial after rescaling.
\end{Conj}

Conjecture~\ref{conj:d_can_basis_reachability} can be generalized
as the following, which implies Conjecture~\ref{conj:analog_leclerc}
by Theorem~\ref{thm:weaker_analog_leclerc}.

\begin{Conj}[multiplicative reachability conjecture]
Let $\can$ denote a common triangular basis. If $b\in\can$ is real
$($i.e., $b^{2}\in\can)$, then it corresponds to a localized quantum
cluster monomial.
\end{Conj}

\begin{Rem}[reachability conjectures]
When the cluster algebra admits an additive categorification by triangulated
categories (cluster categories), we often expect that the rigid objects
(objects with vanishing self-extensions) correspond to the (quantum)
cluster monomials. If so, such objects can be constructed from the
initial cluster tilting objects via (categorical) mutations. Let us
call such an expectation the additive reachability conjecture. This
conjecture is not true for a general cluster algebra because the cluster
algebra seems too small for the cluster category.

When the cluster algebra admits a monoidal categorification by monoidal
categories, we similarly expect that the real simple objects correspond
to the (quantum) cluster monomials (see~\cite{HernandezLeclerc09}).
If so, such objects can be constructed from the an initial collection
of real simple objects via (categorical) mutations. Let us call such
an expectation the multiplicative reachability conjecture.
Conjecture~\ref{conj:d_can_basis_reachability} is related to the special case
for $\qO[N_{-}(w)]$.

We also conjecture an equivalence between the additive reachability
conjecture and the multiplicative reachability conjecture, which can
be viewed as an analog of the open orbit conjecture \cite[Conjecture 18.1]{GeissLeclercSchroeer10}.
See \cite[Section 1]{Nakajima09} for a comparison between additive
categorification and monoidal categorification.

All these conjectures are largely open.
\end{Rem}

\subsection*{Acknowledgements}

The author thanks the referees for helpful suggestions. He also thanks
an anonymous referee for informing him of the work~\cite{kashiwara2020monoidal}.
The author was supported by the National Natural Science Foundation
of China (Grant No.~11701365).

\pdfbookmark[1]{References}{ref}
\LastPageEnding


\begin{thebibliography}{99}
\footnotesize\itemsep=0pt

\bibitem{BerensteinFominZelevinsky05}
Berenstein A., Fomin S., Zelevinsky A., Cluster algebras. {III}.~{U}pper bounds
 and double {B}ruhat cells, \href{https://doi.org/10.1215/S0012-7094-04-12611-9}{\textit{Duke Math.~J.}} \textbf{126} (2005), 1--52,
 \href{https://arxiv.org/abs/math.RT/0305434}{arXiv:math.RT/0305434}.

\bibitem{BerensteinZelevinsky05}
Berenstein A., Zelevinsky A., Quantum cluster algebras, \href{https://doi.org/10.1016/j.aim.2004.08.003}{\textit{Adv. Math.}}
 \textbf{195} (2005), 405--455, \href{https://arxiv.org/abs/math.QA/0404446}{arXiv:math.QA/0404446}.

\bibitem{cautis2019cluster}
Cautis S., Williams H., Cluster theory of the coherent {S}atake category,
 \href{https://doi.org/10.1090/jams/918}{\textit{J.~Amer. Math. Soc.}} \textbf{32} (2019), 709--778,
 \href{https://arxiv.org/abs/1801.08111}{arXiv:1801.08111}.

\bibitem{DerksenWeymanZelevinsky09}
Derksen H., Weyman J., Zelevinsky A., Quivers with potentials and their
 representations {II}: applications to cluster algebras, \href{https://doi.org/10.1090/S0894-0347-10-00662-4}{\textit{J.~Amer.
 Math. Soc.}} \textbf{23} (2010), 749--790, \href{https://arxiv.org/abs/0904.0676}{arXiv:0904.0676}.

\bibitem{FockGoncharov06a}
Fock V.V., Goncharov A.B., Moduli spaces of local systems and higher
 {T}eichm\"uller theory, \href{https://doi.org/10.1007/s10240-006-0039-4}{\textit{Publ. Math. Inst. Hautes \'Etudes Sci.}}
 (2006), 1--211, \href{https://arxiv.org/abs/math.AG/0311149}{arXiv:math.AG/0311149}.

\bibitem{FockGoncharov09}
Fock V.V., Goncharov A.B., Cluster ensembles, quantization and the dilogarithm,
 \href{https://doi.org/10.24033/asens.2112}{\textit{Ann. Sci. \'Ec. Norm. Sup\'er.~(4)}} \textbf{42} (2009), 865--930,
 \href{https://arxiv.org/abs/math.AG/0311245}{arXiv:math.AG/0311245}.

\bibitem{FominZelevinsky02}
Fomin S., Zelevinsky A., Cluster algebras. {I}.~{F}oundations, \href{https://doi.org/10.1090/S0894-0347-01-00385-X}{\textit{J.~Amer.
 Math. Soc.}} \textbf{15} (2002), 497--529, \href{https://arxiv.org/abs/math.RT/0104151}{arXiv:math.RT/0104151}.

\bibitem{FominZelevinsky07}
Fomin S., Zelevinsky A., Cluster algebras. {IV}.~{C}oefficients,
 \href{https://doi.org/10.1112/S0010437X06002521}{\textit{Compos. Math.}} \textbf{143} (2007), 112--164,
 \href{https://arxiv.org/abs/math.RA/0602259}{arXiv:math.RA/0602259}.

\bibitem{GeissLeclercSchroeer10}
Gei{\ss} C., Leclerc B., Schr\"oer J., Kac--{M}oody groups and cluster
 algebras, \href{https://doi.org/10.1016/j.aim.2011.05.011}{\textit{Adv. Math.}} \textbf{228} (2011), 329--433,
 \href{https://arxiv.org/abs/1001.3545}{arXiv:1001.3545}.

\bibitem{GeissLeclercSchroeer11}
Gei{\ss} C., Leclerc B., Schr\"oer J., Cluster structures on quantum coordinate
 rings, \href{https://doi.org/10.1007/s00029-012-0099-x}{\textit{Selecta Math. (N.S.)}} \textbf{19} (2013), 337--397,
 \href{https://arxiv.org/abs/1104.0531}{arXiv:1104.0531}.

\bibitem{goodearl2020integral}
Goodearl K., Yakimov M., Integral quantum cluster structures,
 \href{https://arxiv.org/abs/2003.04434}{arXiv:2003.04434}.

\bibitem{GY13}
Goodearl K.R., Yakimov M.T., Quantum cluster algebra structures on quantum
 nilpotent algebras, \href{https://doi.org/10.1090/memo/1169}{\textit{Mem. Amer. Math. Soc.}} \textbf{247} (2017),
 vii+119~pages.

\bibitem{gross2013birational}
Gross M., Hacking P., Keel S., Birational geometry of cluster algebras,
 \href{https://doi.org/10.14231/AG-2015-007}{\textit{Algebr. Geom.}} \textbf{2} (2015), 137--175, \href{https://arxiv.org/abs/1309.2573}{arXiv:1309.2573}.

\bibitem{gross2018canonical}
Gross M., Hacking P., Keel S., Kontsevich M., Canonical bases for cluster
 algebras, \href{https://doi.org/10.1090/jams/890}{\textit{J.~Amer. Math. Soc.}} \textbf{31} (2018), 497--608,
 \href{https://arxiv.org/abs/1411.1394}{arXiv:1411.1394}.

\bibitem{Hernandez02}
Hernandez D., Algebraic approach to {$q,t$}-characters, \href{https://doi.org/10.1016/j.aim.2003.07.016}{\textit{Adv. Math.}}
 \textbf{187} (2004), 1--52, \href{https://arxiv.org/abs/math.QA/0212257}{arXiv:math.QA/0212257}.

\bibitem{HernandezLeclerc09}
Hernandez D., Leclerc B., Cluster algebras and quantum affine algebras,
 \href{https://doi.org/10.1215/00127094-2010-040}{\textit{Duke Math.~J.}} \textbf{154} (2010), 265--341, \href{https://arxiv.org/abs/0903.1452}{arXiv:0903.1452}.

\bibitem{Kang2018}
Kang S.J., Kashiwara M., Kim M., Oh S.-J., Monoidal categorification of cluster
 algebras, \href{https://doi.org/10.1090/jams/895}{\textit{J.~Amer. Math. Soc.}} \textbf{31} (2018), 349--426,
 \href{https://arxiv.org/abs/1801.05145}{arXiv:1801.05145}.

\bibitem{Kashiwara90}
Kashiwara M., Bases cristallines, \textit{C.~R.~Acad. Sci. Paris S\'er.~I
 Math.} \textbf{311} (1990), 277--280.

\bibitem{Kas:crystal}
Kashiwara M., On crystal bases of the {$Q$}-analogue of universal enveloping
 algebras, \href{https://doi.org/10.1215/S0012-7094-91-06321-0}{\textit{Duke Math.~J.}} \textbf{63} (1991), 465--516.

\bibitem{kashiwara2019laurent}
Kashiwara M., Kim M., Laurent phenomenon and simple modules of quiver {H}ecke
 algebras, \href{https://doi.org/10.1112/s0010437x19007565}{\textit{Compos. Math.}} \textbf{155} (2019), 2263--2295,
 \href{https://arxiv.org/abs/1811.02237}{arXiv:1811.02237}.

\bibitem{kashiwara2020monoidal}
Kashiwara M., Kim M., Oh S.-J., Park E., Monoidal categorification and quantum
 affine algebras, \href{https://doi.org/10.1112/s0010437x20007137}{\textit{Compos. Math.}} \textbf{156} (2020), 1039--1077,
 \href{https://arxiv.org/abs/1910.08307}{arXiv:1910.08307}.

\bibitem{Keller08Note}
Keller B., Cluster algebras, quiver representations and triangulated
 categories, in Triangulated Categories, \textit{London Math. Soc. Lecture
 Note Ser.}, Vol.~375, Cambridge Univetsity Press, Cambridge, 2010, 76--160,
 \href{https://arxiv.org/abs/0807.1960}{arXiv:0807.1960}.

\bibitem{KhovanovLauda08}
Khovanov M., Lauda A.D., A diagrammatic approach to categorification of quantum
 groups.~{I}, \href{https://doi.org/10.1090/S1088-4165-09-00346-X}{\textit{Represent. Theory}} \textbf{13} (2009), 309--347,
 \href{https://arxiv.org/abs/0803.4121}{arXiv:0803.4121}.

\bibitem{KhovanovLauda08:III}
Khovanov M., Lauda A.D., A categorification of quantum {${\rm sl}(n)$},
 \href{https://doi.org/10.4171/QT/1}{\textit{Quantum Topol.}} \textbf{1} (2010), 1--92, \href{https://arxiv.org/abs/0807.3250}{arXiv:0807.3250}.

\bibitem{Kimura10}
Kimura Y., Quantum unipotent subgroup and dual canonical basis,
 \href{https://doi.org/10.1215/21562261-1550976}{\textit{Kyoto~J. Math.}} \textbf{52} (2012), 277--331, \href{https://arxiv.org/abs/1010.4242}{arXiv:1010.4242}.

\bibitem{KimuraQin14}
Kimura Y., Qin F., Graded quiver varieties, quantum cluster algebras and dual
 canonical basis, \href{https://doi.org/10.1016/j.aim.2014.05.014}{\textit{Adv. Math.}} \textbf{262} (2014), 261--312,
 \href{https://arxiv.org/abs/1205.2066}{arXiv:1205.2066}.

\bibitem{Leclerc03}
Leclerc B., Imaginary vectors in the dual canonical basis of {$U_q({\mathfrak
 n})$}, \href{https://doi.org/10.1007/BF03326301}{\textit{Transform. Groups}} \textbf{8} (2003), 95--104,
 \href{https://arxiv.org/abs/math.QA/0202148}{arXiv:math.QA/0202148}.

\bibitem{Lusztig90}
Lusztig G., Canonical bases arising from quantized enveloping algebras,
 \href{https://doi.org/10.2307/1990961}{\textit{J.~Amer. Math. Soc.}} \textbf{3} (1990), 447--498.

\bibitem{Lusztig91}
Lusztig G., Quivers, perverse sheaves, and quantized enveloping algebras,
 \href{https://doi.org/10.2307/2939279}{\textit{J.~Amer. Math. Soc.}} \textbf{4} (1991), 365--421.

\bibitem{Lusztig96}
Lusztig G., Total positivity in reductive groups, in Lie Theory and Geometry,
 \textit{Progr. Math.}, Vol.~123, \href{https://doi.org/10.1007/978-1-4612-0261-5_20}{Birkh\"auser Boston}, Boston, MA, 1994,
 531--568.

\bibitem{Nakajima04}
Nakajima H., Quiver varieties and {$t$}-analogs of {$q$}-characters of quantum
 affine algebras, \href{https://doi.org/10.4007/annals.2004.160.1057}{\textit{Ann. of Math.}} \textbf{160} (2004), 1057--1097,
 \href{https://arxiv.org/abs/math.QA/0105173}{arXiv:math.QA/0105173}.

\bibitem{Nakajima09}
Nakajima H., Quiver varieties and cluster algebras, \href{https://doi.org/10.1215/0023608X-2010-021}{\textit{Kyoto~J. Math.}}
 \textbf{51} (2011), 71--126, \href{https://arxiv.org/abs/0905.0002}{arXiv:0905.0002}.

\bibitem{qin2017triangular}
Qin F., Triangular bases in quantum cluster algebras and monoidal
 categorification conjectures, \href{https://doi.org/10.1215/00127094-2017-0006}{\textit{Duke Math.~J.}} \textbf{166} (2017),
 2337--2442, \href{https://arxiv.org/abs/1501.04085}{arXiv:1501.04085}.

\bibitem{qin2019bases}
Qin F., Bases for upper cluster algebras and tropical points,
 \href{https://arxiv.org/abs/1902.09507}{arXiv:1902.09507}.

\bibitem{qin2020bases}
Qin F., Dual canonical bases and quantum cluster algebras, \href{https://arxiv.org/abs/2003.13674}{arXiv:2003.13674}.

\bibitem{Rouquier08}
Rouquier R., 2-{K}ac--{M}oody algebras, \href{https://arxiv.org/abs/0812.5023}{arXiv:0812.5023}.

\bibitem{Tran09}
Tran T., {$F$}-polynomials in quantum cluster algebras, \href{https://doi.org/10.1007/s10468-010-9226-6}{\textit{Algebr.
 Represent. Theory}} \textbf{14} (2011), 1025--1061, \href{https://arxiv.org/abs/0904.3291}{arXiv:0904.3291}.

\bibitem{VaragnoloVasserot03}
Varagnolo M., Vasserot E., Perverse sheaves and quantum {G}rothendieck rings,
 in Studies in Memory of {I}ssai {S}chur ({C}hevaleret/{R}ehovot, 2000),
 \textit{Progr. Math.}, Vol.~210, \href{https://doi.org/10.1007/978-1-4612-0045-1_13}{Birkh\"auser Boston}, Boston, MA, 2003,
 345--365, \href{https://arxiv.org/abs/math.QA/0103182}{arXiv:math.QA/0103182}.

\end{thebibliography}
\end{document}